\documentclass{amsart}
\usepackage{amsmath}
\usepackage{amsfonts}
\usepackage{amssymb,mathrsfs}
\usepackage{hyperref}
\usepackage{xcolor}
\usepackage{graphicx}
\usepackage{tikz}
\usetikzlibrary{matrix,arrows}
\usepackage[all]{xy}

\graphicspath{{./Images/}}

\newtheorem{Theorem}{Theorem}[section]
\newtheorem{Cor}[Theorem]{Corollary}
\newtheorem{Lemma}[Theorem]{Lemma}
\newtheorem{Proposition}[Theorem]{Proposition}

\newtheorem{Definition}[Theorem]{Definition}
\newtheorem{rem}[Theorem]{Remark}
\newtheorem{example}[Theorem]{Example}
\newtheorem{notation}[Theorem]{Notation}

\newcommand{\R}{\mathbb R}

\newcommand{\N}{\mathbb N}
\newcommand{\Z}{\mathbb Z}

\newcommand{\C}{\mathcal{C}}

\newcommand{\F}{\mathcal{F}}
\newcommand{\p}{\mathcal{D}}
\newcommand{\A}{\mathcal{A}}

\title{On diffeologies for power sets and measures}
\author{Alireza Ahmadi}
\address{Alireza Ahmadi, Department of Mathematics, Yazd University, 89195--741, Yazd, Iran}
\email{ ahmadi@stu.yazd.ac.ir, alirezaahmadi13@yahoo.com}

\author{Jean-Pierre Magnot}
\address{Jean-Pierre Magnot,  CNRS, LAREMA, SFR MATHSTIC, F-49000 Angers, France \\ and \\  Lyc\'ee Jeanne d'Arc,  Avenue de Grande Bretagne,  63000 Clermont-Ferrand, France}

\email{jean-pierr.magnot@ac-clermont.fr}
\begin{document}

	\begin{abstract}
		We consider 
		a differential geometric setting on power sets and Borel algebras. Our chosen framework 
		is	based on diffeologies, and we make a link between the various diffeological structures that we propose, having in mind set-valued maps, relations, set-valued gradients, differentiable measures, and shape analysis. This work intends to 
		establish rigorous properties on sample diffeologies that seem of interest to us.
	\end{abstract}
	
	\maketitle
	
	\tableofcontents
	\section*{Introduction}
	

	From analysis on shape spaces, smoothness of measures and on spaces of measures to set-valued analysis, the question of differentiability on the power set of a smooth space\footnote{By the power set of a space, we mean the set of all subsets of that space.} remains a difficult problem  for which pragmatic approaches are often used. These techniques have their applications in problems of great importance. For example,  
	many real world problems can be reformulated as shape optimization problems constrained by partial differential equations (PDE). A non-exhaustive list of examples includes: inverse modeling of skin structures \cite{NSSW2015}, electrochemical machining \cite{HL2011}, image restoration and segmentation \cite{HR2004}, aerodynamic shape optimisation \cite{SISG2013} and optimization of interfaces in transmission problems \cite{GLMS2015,Pa2015}. Textbooks on the subjects are, for example, \cite{DZ2001,Sz1992}.   
	Often, the shape space is modelled on a vector space, on which one considers landmark positions \cite{CTCG1995,K1984} in most simple cases. But more refined frameworks, such as morphologies of images \cite{DR2007}, multiphase objects \cite{WBRS2011}, characterictic functions of measurable sets \cite{Z2007}, boundary contours of objects \cite{FJSY2009,LJ2007,WR2009} along the lines of \cite{MH1993}, plane curves \cite{MM2007,MSJ2007} and surfaces in higher dimensional manifolds \cite{BHM2011,KMP2007,MM2005}, are also considered in various approaches. 
	
	From another point of view, and with a field of applications as large as differentiability of shapes,  differentiablility of measures on the domain defined by a Borel algebra, based on and extending the differentiation of functions through their representation via Radon measures, various ways to define differentiability of measures have been defined (see, e.g., \cite{Bog}). In the same way, differential geometric techniques are crucial in the optimization of set-valued functions. 
	Although there are several notions of smoothness for set-valued maps of Euclidean spaces with different approaches, the smoothness of set-valued maps has not been developed on more general smooth spaces such as manifolds.
	
	All these notions are centered on the power set of a smooth space equipped with a ``nice'' and ``suitable'' differential geometric setting that induces, in various ways, a differentiation-like procedure on the power set in question. This differentiation remains in applications an important step for the implementation of the resolution of equations.

	The aim of the present work is to provide a suitable systematic framework for 
	differentiation on the power set of a space, in which the flavor of rigor and adeptness is comforting enough to enable one to think about a general theory adapted to applied frameworks. 
	For this purpose, we work in the framework of diffeology, which was established by J.-M. Souriau \cite{Sou} in the 1980s  and has been developed by P. Iglesias-Zemmour and others in various directions.
	Diffeological spaces form a category for differential calculus and differential geometry, actually one of the widest categories in which the calculus of variations is well-defined, and for which the technical problems remain reasonable. The reader can refer to the book \cite{Igdiff} for a comprehensive introduction to diffeologies, while Fr\"olicher spaces \cite{FK,KM} form an interesting subcategory of the category of diffeological spaces.   Diffeological spaces \cite{Igdiff,Sou} as well as Fr\"olicher spaces \cite{FK,KM} define two categories of generalized frameworks for differential geometry, where the presence of atlases is not needed, and the necessary properties for the classical calculus of variations are fulfilled. 
	
	The contents of the paper are as follows. 
	We recall basics on diffeologies and Fr\"olicher spaces, at least those properties that are crucial to understanding the framework we develop for power sets in Section \ref{Prel}. Diffeologies on a set $X$ are based on the notion of plots, that is, maps with images in $X$ which are decided to be smooth. The set of such smooth maps that agree with some minimal properties is called a diffeology on $X $. 
	
	After that, we provide and study in detail samples of diffeologies on power sets with desirable properties in Sections \ref{wpsdiff}, \ref{upsdiff} and \ref{spsdiff}. These are new refinements of the power set diffeology described in Iglesias-Zemmour's book \cite{Igdiff} on the power set
	$ \mathfrak{P}(X) $ of a diffeological space $ X $, which we call the weak, the union, and the strong power set diffeologies, and are derived from the diffeological space $ X $ with some certain properties.
	Comparing these diffeologies, we get the following inclusions:
	\begin{center}
		Strong diffeology $\quad \hookrightarrow\quad $	Union diffeology $ \quad\hookrightarrow\quad $ Weak diffeology
	\end{center}
	We discuss their geometrical properties, especially we prove that a diffeological space is ``strongly" embedded into its power set endowed with either the union or the strong power set diffeologies.
	
	Having diffeological structures on power sets, we are easily able to treat smooth set-valued maps as usual smooth maps in the diffeological setting.
	From a topological perspective, we show that smooth set-valued maps are lower semi-continuous with respect to the D-topology.
	Similar to continuous selection problems (see, e.g., \cite{Mi1956}), we consider the smooth versions of these problems and explore their relationships with the smoothness of a set-valued map. In particular, on a manifold, the smoothness of a set-valued map is equivalent to solving the smooth selection problem over an open cover. In the case of Euclidean spaces, this result is compatible with \cite[Theorem 2.2.1]{SZ}.
	It is well-known that there is a  one-to-one correspondence between set-valued maps and relations.
	We suggest notions of smooth relations as those whose corresponding set-valued maps are smooth.
	Moreover, we consider the space of smooth relations between diffeological spaces and compare the induced diffeologies on it.
	
	In Section \ref{ppsdiff}, we describe   
	so-called projectable diffeologies on power sets, developed in the same spirit as the $Diff-$diffeology defined in \cite{Ma2020-3}, that is, the plots are derived from smooth families of functions acting on $X$. We observe that:
	\begin{center}
		Locally projectable diffeology $\quad \hookrightarrow\quad $	Union diffeology
	\end{center}
	
	In these various definitions, which may appear as natural from one viewpoint or another, we do not find so easily diffeologies for which Boolean operations are smooth (we call them Boolean diffeologies). 
	In Section \ref{bdiff}, we prove the existence of diffeologies on the power set $\mathfrak{P}(X)$ for which the Boolean operations are smooth. They will be discussed  in more detailed examples in other upcoming sections.  
	
	We also show that it is possible to fit with some existing frameworks for differentiation 
	in   our setting. In Sections \ref{diffmeas} and \ref{Fom}, we analyze how diffeologies on a Borel algebra may encode differentiability, or smoothness of measures in it. We first analyze in Section \ref{diffmeas} the 
	existence  of a Boolean diffeology, which we show to be always defined, but also not-so natural in the concrete example of measures on a smooth manifold   since we prove (Remark \ref{vaguediff}) that a natural intuitive diffeology is not  a priori Boolean in this framework.  Then, we show that we can recover Fomin differentiability from the viewpoint of a ``directional'' diffeology in Section \ref{Fom}. 
	In Section \ref{measdiff}, we define a diffeology on the space of measures. This example is motivated by the counter-example of Remark \ref{vaguediff} of Section \ref{diffmeas}, and requires more deep and more complete investigations that must be developed elsewhere. 
	This diffeology  on measures defined in Section \ref{measdiff}  
	is different from the diffeology defined in \cite{Ma2020-3}, and not only on the space of probabilities but on the full space of (positive), non necessarily finite, measures. 
	
	We finish with a possible framework for set-valued maps, and show how Hadamard derivatives of functionals on shape analysis fit with a directional derivative of the globally projectable diffeology in Sections \ref{stvdiff} and \ref{diffshp}.



\section{Preliminaries}\label{Prel}

\subsection{On diffeologies} 
We review here the basics of the theory of diffeological spaces, especially, their definitions, categorical properties, as well as their induced topology. {{} The main idea of diffeologies (and Fr\"olicher spaces defined shortly after) is to replace the atlas of a classical manifold with other intrinsic objects that enable to define smoothness of mappings in a safe way, considering manifolds as a restricted class of examples. Many such settings have been developed independently \cite{Sta}. We choose these two settings because they carry nice properties such as cartesian closedness, the necessary fundamental properties of, e.g., calculus of variations, and also because they are very easy to use in a differential geometric way of thinking. The fundamental idea of these two settings is based on defining families of smooth maps, with mild conditions on them which ensure technical features of interest. }

\begin{Definition}[Diffeology] \label{d:diffeology}
	Let $X$ be a set.  A \textbf{parametrization} in $X$ is a
	map of sets
	$P \colon U \to X$, where $U$ is an open subset of Euclidean space (no fixed dimension).  A \textbf{diffeology} $\p$ on $X$ is a set of
	parametrizations in $X$ satisfying the following three conditions:
	\begin{enumerate}
		\item (Covering) For every $x\in X$ and every non-negative integer
		$n$, the constant map $P\colon \R^n\to\{x\}\subset X$ is in
		$\p$.
		\item (Locality) Let $P\colon U\to X$ be a parametrization such that for
		every $r\in U$ there exists an open neighborhood $V\subset U$ of $r$
		satisfying $P|_V\in\p$. Then $P\in\p$.
		\item (Smooth Compatibility) Let $(P\colon U\to X)\in\p$.
		Then for every $n$, every open subset $V\subset\R^n$, and every
		smooth map $F\colon V\to U$, we have $P\circ F\in\p$.
	\end{enumerate}
	A set $X$ equipped with a diffeology $\p$ is called a
	\textbf{diffeological space}, and is {{} denote}d by $(X,\p)$.
	When the diffeology is understood, we will drop the symbol $\p$.
	The parametrizations $P\in\p$ are called \textbf{plots} in the space $X$.
\end{Definition}
\begin{Definition} 
	A collection $\mathcal{D}$ of parametrizations in a set $X$ satisfying the covering and smooth compatibility conditions is said to be a  \textbf{prediffeology} on $X$.
	If $\mathcal{D}$ fulfills the covering condition, it is a \textbf{parametrized cover} of $X$.
\end{Definition} 
\begin{Definition}
	Let $ X $ be any set, and let $\p$ and $\p'$ be
	diffeologies on $X$. If $\p\subset \p'$, then  $\p$ is \textbf{finer} that $\p'$, or equivalently, $\p'$ is \textbf{coarser} that $\p $. 
\end{Definition}

\begin{example}
	Let $ X $ be any set.
	The set of the locally constant parametrizations in  $ X $ is a diffeology on $ X $ called the \textbf{discrete
		diffeology}.
	The set of all parametrizations in a set $ X $ is a diffeology on $ X $ called the \textbf{indiscrete} or \textbf{coarse
		diffeology}.
	It is trivial that the discrete
	diffeology is the finest diffeology and the indiscrete diffeology is the coarsest diffeology on $ X $, and any other diffeology on $ X $ is between them.
\end{example}
\begin{Definition} 
	A family $\lbrace P_i: U_i\rightarrow X\rbrace_{i\in J}$ of parametrizations defined on open subsets of $ \mathbb{R}^n $  is \textbf{compatible} if
	$P_i|_{U_i\cap U_j}=P_j|_{U_i\cap U_j}$, for all $i, j\in J$. 
	For such a family, the parametrization 
	$P:\bigcup_{i\in J} U_i\rightarrow X$ given by $P(r)=P_i(r)$ for $r\in U_i$, 
	is said to be the \textbf{supremum} of the family. 
\end{Definition} 
\begin{rem}
	The locality condition of diffeology in Definition \ref{d:diffeology} is equivalent to saying that the supremum of any compatible family of plots is itself a plot.
\end{rem}

\begin{Definition}
	Let  $X$ be a set and let $\mathcal{C}$ be a set of
	parametrizations in $X$ satisfying the covering condition in Definition \ref{d:diffeology}.
	The \textbf{diffeology generated} by $ \mathcal{C} $, denoted by $\langle\mathcal{C}\rangle$, is the set of  parametrizations $ P $ 
	that are the supremum of a compatible family
	$ \lbrace P_i\rbrace_{i\in J} $ of parametrizations in $ X $ in the form $ P_i=Q_i\circ F_i $,
	where $ Q_i $ is an element of $\mathcal{C}$ and $ F $ is a smooth map between domains.
	For a diffeological space $ (X,\mathcal{D}) $,
	a \textbf{covering generating family} is a parametrized cover $\mathcal{C}$ of $ X $  generating the diffeology of the space, i.e., $\langle\mathcal{C}\rangle=\mathcal{D}$.
	Denote by $ \mathsf{CGF}(X) $ the collection of all covering generating families of the space $ X $.
\end{Definition}

\begin{Definition}
	The \textbf{dimension} $ \mathrm{dim}(X) $ of a  diffeological space $ X $ is a nonnegative integer that is defined by
	\begin{center}
		$ \mathrm{dim}(X)= ~~~\inf_{\mathcal{G} \in\mathsf{CGF}(X)} ~~\mathrm{dim}(\mathcal{G}), $
	\end{center}
	where
	\begin{center}
		$ \mathrm{dim}(\mathcal{G})=~~~\sup_{P\in\mathcal{G}} ~~\mathrm{dim}\big{(}\mathrm{dom}(P)\big{)} $
	\end{center}
	for each $ \mathcal{G} \in\mathsf{CGF}(X) $.
	If $ X $ has no covering generating family with finite dimension, we write $ \mathrm{dim}(X)=\infty $.
\end{Definition} 

\begin{Definition}[Diffeologically Smooth Maps]\label{d:diffeolmap}
	Let $(X,\p_X)$ and $(Y,\p_Y)$ be two diffeological
	spaces, and let $F \colon X \to Y$ be a map.  Then we say that $F$ is
	\textbf{diffeologically smooth} if for any plot $P \in \p_X$,
	$$F \circ P \in \p_Y.$$
	Any smooth map with a smooth inverse is called a \textbf{diffeomorphisms}.
\end{Definition}

Diffeological spaces with diffeologically smooth maps form a category. This category is complete and co-complete, and forms a quasi-topos (see \cite{BH}).
\begin{Definition}
	Let $X$ and $Y$ be diffeological spaces.
	The \textbf{functional diffeology} on the set $C^{\infty}(X,Y)$ of all smooth maps from $X$ to $Y$  is
	given by the following condition: 
	A parametrization $ Q:V\rightarrow C^{\infty}(X,Y)$ is a plot for the functional diffeology if and only if for every plot $P:U\rightarrow X$, the parametrization
	$ Q\circledcirc P:V\times U\rightarrow Y$
	with
	$ (Q\circledcirc P)(r,s)= Q(r)\big{(}P(s)\big{)} $
	is a plot in $ Y $. 
\end{Definition} 

Indeed, the functional diffeology on $ C^{\infty}(X,Y) $ is the coarsest diffeology for which
$ \mathrm{ev}: C^{\infty}(X,Y)\times X\rightarrow Y $ defined by $ \mathrm{ev}(f,x) $ to  $ f(x) $ is smooth.

\begin{Definition}
	\cite{Sou,Igdiff} Let $(X',\p)$ be a diffeological space,
	and let $X$ be a set. Let $f:X\rightarrow X'$ be a map.
	We define $f^*(\p)$ the \textbf{pull-back diffeology} as {{} $$f^*( \p)= \left\{ P: U \rightarrow X \, |   f \circ P \in \p \right\}. $$ }
\end{Definition}
\begin{Definition}
Let $(X,\p_X)$ and $(Y,\p_Y)$
be two diffeological spaces.  An injective map
$f : X \rightarrow Y$
is  called  an
\textbf{induction}
if  $\p_Y = f^*(\p_X).$ 
\end{Definition}

\begin{Definition} \cite{Sou,Igdiff} Let $(X,\p)$ be a diffeological space,
and let $X'$ be a set. Let $f:X\rightarrow X'$ be a map.
We define $f_*(\p)$ the \textbf{push-forward diffeology} as $ \langle f \circ \p\rangle $, which is the coarsest diffology
on $X'$ containing $f \circ \p=\{f\circ P\mid P\in\p\}.$ 
\end{Definition}  

\begin{Definition}
Let $(X,\p_X)$ and $(Y,\p_Y)$
be two diffeological spaces.  A map
$f : X \rightarrow Y$
is  called  a
\textbf{subduction}
if  $\p_Y = f_*(\p_X).$ 
\end{Definition}

In particular, we have the following constructions.

\begin{Definition}[Product Diffeology]\label{d:diffeol product}
Let $\{(X_i,\p_i)\}_{i\in I}$ be a family of diffeological spaces.  Then the \textbf{product diffeology} $\p$ on $X=\prod_{i\in I}X_i$ contains a parametrization $P\colon U\to X$ as a plot if for every $i\in I$, the map $\pi_i\circ P$ is in $\p_i$.  Here, $\pi_i$ is the canonical projection map $X\to X_i$. 
\end{Definition}

In other words, in last definition, $\p = \bigcap_{i \in I} \pi_i^*(\p_i)$ and each $\pi_i$ is a subduction.

\begin{Definition}[Subset Diffeology]\label{d:diffeol subset}
Let $(X,\p)$ be a diffeological space, and let $Y\subset X$.  Then $Y$ comes equipped with the \textbf{subset diffeology}, which is the set of all plots in $\p$ with image in $Y$.
\end{Definition}

\begin{Definition}
Every diffeological space $ X $ has a natural topology called the D-\textbf{topology} in which a subset of $X$ is D-\textbf{open} if its preimage by any plot is open. 

\end{Definition}

Any smooth map is D-continuous, that is, continuous with respect to the D-topology \cite[\S 2.9]{Igdiff}.

\begin{notation}
We recall that $\N^* = \{n \in \N \, | \, n \neq 0\}$ and that $\forall m \in \N^*, \N_m = \{1,...,m\} \subset \N.$
\end{notation}
If $ M $ is a smooth manifolds, finite or infinite dimensional, modelled on a complete locally convex topological vector space, we define the \textbf{nebulae diffeology}
$$\p_\infty(M) = \left\{ P \in C^\infty(U,M) \hbox{ (in the usual sense) }| U \hbox{ is open in } \R^d, d \in \N^* \right\}.$$

\subsection{Diffeological submersions, immersions, and \'{e}tale maps}
We here briefly recall the needed definitions and results without proofs from \cite{ARA}.
\begin{Definition} 
A (not necessarily surjective) smooth map $ f:X\rightarrow Y $ between diffeological spaces is a \textbf{weak subduction} if for any plot $ P:U\rightarrow Y $  and $ r_0\in U $, if $ P(r_0)\in f(X) $, then there exists at least one local lift plot $ L:V\rightarrow X $ defined on an open
neighborhood $ V\subset U $ of $ r_0 $ with $ f\circ L=P|_V $.
In this situation, $ f(X) $ is a D-open subset of $ Y $.
\end{Definition}

\begin{Definition}
We call a smooth map $ f:X\rightarrow Y $ between diffeological spaces a \textbf{submersion} if for each $ x_0 $ in $ X $, there exists a smooth local section $ \sigma:O\rightarrow X $ of $ f $
passing through $ x_0 $ defined on a D-open subset $ O\subseteq Y $ such that $ f\circ\sigma(y)=y $ for all $ y\in O $.
A smooth map $ f:X\rightarrow Y $  is said to be a \textbf{diffeological submersion} if
the pullback  of $ f $ by every plot in $ Y $ is a submersion.
A \textbf{local subduction} is a surjective diffeological submersion.
\end{Definition}

\begin{Definition} 
A smooth map $ f:X\rightarrow Y $ between diffeological spaces is an \textbf{immersion} if for each $ x_0 $ in $ X $, there exist a D-open neighborhood $ O\subseteq X $ of the point $ x_0 $, a D-open neighborhood $ O'\subseteq Y $ of the set $ f(O) $, and a smooth map $ \rho:O'\rightarrow X $ such that 
$ \rho\circ f(x)=x $ for all $ x\in O $. 
A smooth map $ f:X\rightarrow Y $  is a \textbf{diffeological immersion} if
for any plot $ P:U\rightarrow Y $ in $ Y $,
for each $ (r_0,x_0) $ in $ P^*X $, there exist a D-open neighborhood $ O $ of $ (r_0,x_0) $ in $ P^*X $, an open neighborhood $ V\subseteq U $ of $ P^*f(O) $ and a smooth map $ \rho:V\rightarrow U\times X $ such that $ \rho\circ P^*f(r,x)=(r,x) $ for all $ (r,x)\in O $.
\begin{displaymath}
\xymatrix{
	U\times X\ar@/^0.4cm/[dr]^{\Pr_2} & \\
	P^*X \ar[u]\ar[r]^{P_{\#}}\ar[d]^{P^*f} & X\ar[d]^{f} \\
	V\ar@/^0.6cm/[uu]^{\rho}\ar[r] \subseteq U \ar[r]^{P} & Y  }
\end{displaymath}
\end{Definition}

\begin{Proposition}\label{p:immr}\cite{ARA}
Suppose we are given 
a commutative diagram of smooth maps as the following:
$$
\xymatrix{
X \ar[dr]_{g} \ar[rr]^{f} & &Y \ar[dl]^{h}\\
& Z  & 
}
$$
If $ g $ is a diffeological immersion, then  $ f $ is a diffeological immersion.
In particular, if $ f:X\rightarrow (Y,\mathcal{D}) $ is a diffeological immersion, and 
$ \mathrm{id}:(Y,\mathcal{D}')\rightarrow (Y,\mathcal{D}) $	is smooth, then $ f:X\rightarrow (Y,\mathcal{D}') $ is a diffeological immersion.
\end{Proposition}

\begin{Definition} 
A map $ f:X\rightarrow Y $ between diffeological spaces is called a \textbf{strong embedding} if it is an induction, a diffeological immersion, and a D-embedding (i.e., a topological embedding with respect to the D-topology).
\end{Definition}

\begin{Definition}
A map $ f:X\rightarrow Y $ between diffeological spaces is \textbf{\'{e}tale} if for every $ x $ in $ X $, there are D-open neighborhoods $ O\subseteq X $ and $ V\subseteq Y $ of $ x $ and $ f(x) $, respectively, such that $ f|_O:O\rightarrow O' $ is a diffeomorphism.  
A smooth map $ f:X\rightarrow Y $ is a \textbf{diffeological \'{e}tale map} if the pullback $ P^{*}f $ by every plot $ P $ in $ X $ is  \'{e}tale. 
\end{Definition}
\subsection{On Fr\"olicher spaces}
\begin{Definition} $\bullet$ A \textbf{Fr\"olicher} space is a triple
$(X,\F,\C)$ such that

- $\C$ is a set of paths $\R\rightarrow X$,

- A function $f:X\rightarrow\R$ is in $\F$ if and only if for any
$c\in\C$, $f\circ c\in C^{\infty}(\R,\R)$;

- A path $c:\R\rightarrow X$ is in $\C$ (i.e. is a \textbf{contour})
if and only if for any $f\in\F$, $f\circ c\in C^{\infty}(\R,\R)$.

\vskip 5pt $\bullet$ Let $(X,\F,\C)$ and $(X',\F',\C')$ be two
Fr\"olicher spaces, a map $f:X\rightarrow X'$ is \textbf{differentiable}
(=smooth) if and only if one of the following equivalent conditions is fulfilled:
\begin{itemize}
\item $\F'\circ f\circ\C\subset C^{\infty}(\R,\R)$
\item $f \circ \C \subset \C'$
\item $\F'\circ f \subset  \F$ 
\end{itemize}
\end{Definition}

Any family of maps $\F_{g}$ from $X$ to $\R$ generate a Fr\"olicher
structure $(X,\F,\C)$, setting \cite{KM}:

- $\C=\{c:\R\rightarrow X\hbox{ such that }\F_{g}\circ c\subset C^{\infty}(\R,\R)\}$

- $\F=\{f:X\rightarrow\R\hbox{ such that }f\circ\C\subset C^{\infty}(\R,\R)\}.$

One easily see that $\F_{g}\subset\F$. This notion will be useful
in the sequel to describe in a simple way a Fr\"olicher structure.
A Fr\"olicher space carries a natural topology,
which is the pull-back topology of $\R$ via $\F$. In the case of
a finite dimensional differentiable manifold, the underlying topology
of the Fr\"olicher structure is the same as the manifold topology. In
the infinite dimensional case, these two topologies differ very often.

Let us now compare Fr\"olicher spaces with diffeological spaces, with the following diffeology {{}$\p_\infty(\F)$} called "nebulae":
{{}
{Let }$U${ be an open subset of a Euclidean space; } $$\p_\infty(\F)_U=
\coprod_{n\in\N}\{\, f : U \rightarrow X; \, \F \circ f \subset C^\infty(U,\R) \quad \hbox{(in
the usual sense)}\}$$
and 
$$ \p_\infty(\F) = \bigcup_U \p_\infty(\F)_U,$$
where the latter union is extended over all open sets $U \subset \R^n$ for $n \in \N^*.$ 
}
With this construction, we get a natural diffeology when
$X$ is a Fr\"olicher space. In this case, one can easily show the following:
\begin{Proposition} \label{Frodiff} \cite{Ma2006-3} 
Let $(X,\F,\C)$
and $(X',\F',\C')$ be two Fr\"olicher spaces. A map $f:X\rightarrow X'$
is smooth in the sense of Fr\"olicher if and only if it is smooth for
the underlying nebulae diffeologies. \end{Proposition}

Thus, we can also state intuitively:
\vskip 12pt
\begin{tabular}{ccccc}
smooth manifold  & $\Rightarrow$  & Fr\"olicher space  & $\Rightarrow$  & Diffeological space\tabularnewline
\end{tabular}
\vskip 12pt
With this construction, any complete locally convex topological vector space is a diffeological vector space, that is, a vector space for which addition and scalar multiplication is smooth. The same way, any finite or infinite dimensional manifold $X$ has a nebulae diffeology, which fully determines smooth functions from or with values in $X.$ We now finish the comparison of the notions of diffeological and Fr\"olicher 
space following mostly \cite{Ma2006-3,Wa}, see, e.g., \cite{Ma2020-3}:

\begin{Theorem} \label{compl-fro}
Let $(X,\p)$ be a diffeological space. There exists a unique Fr\"olicher structure
$(X, \F_\p, \C_\p)$ on $X$ such that for any Fr\"olicher structure $(X,\F,\C)$ on $X,$ these two equivalent conditions are fulfilled:

(i)  the canonical inclusion is smooth in the sense of Fr\"olicher $(X, \F_\p, \C_\p) \rightarrow (X, \F, \C)$

(ii) the canonical inclusion is smooth in the sense of diffeologies $(X,\p) \rightarrow (X, \p_\infty(\F)).$ 

\noindent Moreover, $\F_\p$ is generated by the family 
$$\F_0=\lbrace f : X \rightarrow \R \hbox{ smooth for the 
usual diffeology of } \R \rbrace.$$
{{} We call \textbf{Fr\"olicher completion} of $\p$ the Fr"olicher structure $(X, \F_\p, \C_\p).$}
\end{Theorem}

\begin{Definition} \cite{Wa}
A \textbf{reflexive} diffeological space is a diffeological space $(X,\p)$ such that $\p = \p_\infty(\F_\p).$
\end{Definition}

\begin{Theorem} \cite{Wa}
The category of Fr\"olicher spaces is exactly the category of reflexive diffeological spaces.
\end{Theorem}

This last theorem allows us to make no difference between Fr\"olicher spaces and reflexive diffeological spaces. 
We shall call them Fr\"olicher spaces, even when working with their underlying diffeologies.
\vskip 12pt A deeper analysis of these implications has been given
in \cite{Wa}. The next remark is inspired on this work and on
\cite{Ma2006-3}; it is based on \cite[p.26, Boman's theorem]{KM}. For this, we have to define the 1 dimensional diffeology $\p_1(\F),$ also called ``spaghetti diffeology'' in \cite{CW2022}, made of plots $ P \in \p_\infty(\F)$ which factor smoothly through $\R:$
$$ P : U\subset \R^n \rightarrow \R \rightarrow X.$$
This is also the diffeology minimal for inclusion which contains the set of contours $\C.$  
\begin{rem}
We notice that the set of contours $\C$ of the Fr\"olicher space
$(X,\F,\C)$ \textbf{does not} give us a diffeology, because a diffelogy
needs to be stable under restriction of domains. In the case of paths in
$\C$ the domain is always $\R.$ However, $\C$ defines a ``minimal diffeology''
$\p_1(\F)$ whose plots are smooth parameterizations which are locally of the
type $c \circ g,$ where $g \in \p_\infty(\R)$  and $c \in \C.$ Within this setting,
a  map $f : (X,\F,\C) \rightarrow (X',\F',\C')$ is smooth if and only if it is smooth  $(X,\p_\infty(\F)) \rightarrow (X',\p_\infty(\F')) $ or equivalently smooth  .$(X,\p_1(\F)) \rightarrow (X',\p_1(\F')) $ 
\end{rem}
We apply the results on product diffeologies to the case of Fr\"olicher spaces and we derive very easily, (compare with, e.g., \cite{KM}) the following:

\begin{Proposition} \label{prod2} Let $(X,\F,\C)$
and $(X',\F',\C')$ be two Fr\"olicher spaces equipped with their natural
diffeologies $\p$ and $\p'$ . There is a natural structure of Fr\"olicher space
on $X\times X'$ which contours $\C\times\C'$ are the 1-plots of
$\p\times\p'$. \end{Proposition}

We can even state the result above for the case of infinite products;
we simply take cartesian products of the plots or of the contours.
We also remark that given an algebraic structure, we can define a
corresponding compatible diffeological structure. For example, a
$\R-$vector space equipped with a diffeology is called a
diffeological vector space if addition and scalar multiplication
are smooth (with respect to the canonical diffeology on $\R$), see \cite{Igdiff}. An
analogous definition holds for Fr\"olicher vector spaces. Other
examples will arise in the rest of the text.

\begin{rem} \label{comp}
Fr\"olicher, $c^\infty$ and G\^ateaux smoothness are the same notion
if we restrict to a Fr\'echet context, see \cite[Theorem 4.11]{KM}.
Indeed, for a smooth map $f : (F, \p_1(F)) \rightarrow \R$ defined
on a Fr\'echet space with its 1-dimensional diffeology, we have
that $\forall (x,h) \in F^2,$ the map $t \mapsto f(x + th)$ is
smooth as a classical map in $\C^\infty(\R,\R).$ And hence, it is
G\^ateaux smooth. The converse is obvious.
\end{rem}

\subsection{The tangent space of interest for us}
The question of tangent spaces is not yet solved for diffeological spaces. Indeed, even if there exists actually a shared notion of cotangent space of a diffeological space, the multiple possible generalizations of the notion of tangent space to a diffeological space appear as non-equivalent and many of them are well-motivated by their applications. 
A non-exhaustive list of tangent spaces is the following:
\begin{enumerate}
\item \label{iT} The tangent cone defined by \cite{Les} where tangent elements are germs of paths on the diffeological space, identified through tangent plots.
\item  The tangent cone defined by \cite{Ma2013} for Fr\"olicher spaces and extended in \cite{GMW2023} to any diffeological space, where tangent elements are germs of paths on the diffeological space, understood as local derivations of smooth $\R-$valued functions.This tangent cone can be different from the first one. 
\item One can consider linear combinations in each of these tangent cones, along the line of \cite{CW2014} and \cite{GW2022} respectively for each context.
\item the Diff-tangent space defined in \cite{Ma2020-3} for Fr\"olicher spaces, where tangent element are evaluations at one point of tangent vector fields, understood as elements of the internal tangent space at the identity of the group of diffeomorphisms.
\end{enumerate}
For an extensive review of these constructions, we refer to \cite{GMW2023}. 
Since we need to evaluate differentials on tangent vectors, we choose the definition (\ref{iT}) for the tangent space.
More precisely, 
we consider the set of all paths that maps $ 0 $ to $ x $ and identify two of these paths with each other if they coincide on tangent spaces of plots.
The domain $\mathrm{dom}(P)$ of each plot $P$ of $\p$ can be considered as a smooth manifold, and first objects of interest are tangent vectors in the tangent space $T\mathrm{dom}(P).$ They are understood as germs of smooth path $\frac{d\gamma}{dt}|_{t=0}$ where $\gamma \in C^\infty(\R,\mathrm{dom}(P)).$  Let $x \in X$ and let us consider 
$$\mathcal{C}_x = \left\{ \gamma \in C^\infty(\R,X) \, | \, \gamma(0)=x\right\}.$$ For each $P \in \p,$ we also define
$$\mathcal{C}_{x,P} = \left\{\gamma \in C^\infty(\R,\mathrm{dom}(P)) \, | \, P \circ \gamma(0) = x\right\}.$$ This set of smooth paths passing at $x$ enables to define the kinematic set 

$$\mathcal{K}_x = \coprod_{P \in \p} \left\{ \frac{d \gamma}{dt}|_{t=0} \, | \, \gamma \in \C_{x,P} \right\} = \coprod_{p \in \p} \coprod_{x_0 \in p^{-1}(x)} T_{x_0}\mathrm{dom}(P).$$

Therefore, we identify $(X_1,X_2) \in \mathcal{K}_x^2, $ where $X_1 = \frac{d \gamma_1}{dt}|_{t=0} \in T_{x_1}\mathrm{dom}(P_1)$ and  $X_2 = \frac{d \gamma_2}{dt}|_{t=0} \in T_{x_2}\mathrm{dom}(P_2)$ if there exists a $P_{3} \in \p$ and $(\gamma_{3,1},\gamma_{3,2}) \in C_{x,P_3}$ such that
\begin{equation} \label{id-germs}  \left\{\begin{array}{l} \forall i \in \{1,2\}, P_i \circ \gamma_i = P_{3,i} \circ \gamma_{3,i} \\
\frac{d\gamma_{3,1}}{dt}|_{t=0} = \frac{d\gamma_{3,2}}{dt}|_{t=0}  
\end{array}\right.\end{equation}
This identification is reflexive, symmetric, but not transitive, as is shown in the following counter-example:
\begin{example} \label{spagh}
Let $$ X= \{(x,y,z) \in \R^3 \, | \, yz = 0\}.$$
We equip $X$ with its subset diffeology, inherited from the nebulae diffeology of $\R^3.$ Let us conside the paths $$\gamma_1(t) = (t,t^2,0)$$ and $$\gamma_2(t)=(t,0,t^2).$$
The natural intuition (which will be shown to have a defect later in the exposition, for another diffeology on $X$) says that $$\frac{d\gamma_{1}}{dt}|_{t=0} = \frac{d\gamma_{2}}{dt}|_{t=0} = (1,0,0),$$
\emph{but there is no parametrization $P_3$ at $x=(0,0,0)$ that fulfills (\ref{id-germs}).}
Indeed, the diffeology of $X$ is generated by the push-forwards of the plots of $\R^2$ to $X$ by the maps $(x,y) \mapsto (x,y,0) $ and $(x,y \mapsto (x,0,y).$ In order to identify the two germs, one has to consider an intermediate path $\gamma_{1.5}(t) = (t,0,0).$ 

\end{example}
Therefore, we define 

\begin{Definition}
We define the equivalence relation $\sim$ on $\mathcal{K}_x$ as follows: 
$\forall (X_1,X_2) \in \mathcal{K}_x, u_1 \sim u_2$ if and only if one of the two following conditions is fulfilled: 
\begin{enumerate}
\item $$\exists (\gamma_1,\gamma_2) \in \left(\coprod_{P \in \p}\C_{x,P}\right)^2,  \frac{d\gamma_{1}}{dt}|_{t=0} = u_1 \in T\mathrm{dom}(P_1), \frac{d\gamma_{2}}{dt}|_{t=0} = u_2 \in T\mathrm{dom}(P_2)$$ and also $P_{3} \in \p$ and $(\gamma_{3,1},\gamma_{3,2}) \in C_{x,P_3}$ such that condition (\ref{id-germs}) applies. 
\item there exists a finite sequence $(v_1,...,v_k)\in \mathcal{K}_x^k$ such that $v_1 = u_1,$ $v_k = u_2,$ and such that condition (1) applies to $v_i$ and $v_{i+1}$ for each index $i \in \N_{k-1}.$
\end{enumerate} 

\end{Definition}
{\begin{Definition}
The  \textbf{internal}  or \textbf{kinetic tangent cone} of $X$ at $x \in X$ is defined by $${}^iT_xX = \mathcal{K}_x/\sim,$$
The space ${}^iT_xX$ is endowed by the push-forward of the functional diffeology on $\mathcal{C}_x.$
\end{Definition}}
The elements of ${}^iT_xX$ are called \textbf{germs of paths} on $X$ at $x.$ 

\section{The weak power set diffeology}\label{wpsdiff}
Let us begin with the first type of diffeologies on the power set of a diffeological space.
\begin{Definition}
Let  $ X $  be a diffeological space.  The set of parametrizations 
$ P:U\rightarrow\mathfrak{P}(X) $ with the following property is a diffeology on $  \mathfrak{P}(X) $, called the \textbf{weak power set diffeology}:
\begin{enumerate}
\item[\textbf{WPD.}]
For every $ r_0  \in U $, if $ P(r_0)\neq\emptyset $, there exists an open neighborhood $ V \subset U $ of $ r_0  $
and a plot $ \sigma: V \rightarrow X $, called \textbf{local selection plot}, such that $ \sigma(r) \in P(r) $, 
for every $ r \in V $.
\end{enumerate} 
We denote by $ \mathfrak{P}_{w}(X) $  the power set	$  \mathfrak{P}(X) $ endowed with the weak power set diffeology and call it the \textbf{weak power space}.
\end{Definition}

\begin{rem}
The weak power set diffeology differs from the ``minimal power set diffeology" given in \cite[Exercise 32]{Igdiff}.
In fact, the minimal power set diffeology is finer than the weak power set diffeology. For instance, the plot $ P:\mathbb{R}\rightarrow\mathfrak{P}_{w}(X) $ in the weak power space, with $ r\in\mathbb{R}-\{0\}\mapsto X $ and $ 0\mapsto  \emptyset $, is not a plot with respect to the minimal power set diffeology. Because, there is no neighborhood $ V\subset\mathbb{R} $ of $ 0 $ such that $ P|_V $ is constant with the value $ \emptyset $. But if we restrict ourselves to the set $\mathfrak{P}^\star(X)$ of nonempty subsets of $ X $, these two diffeologies coincide.
\end{rem}
\begin{Proposition}\label{pro: weak-immr}
Let  $ X $  be a diffeological space. 
The canonical map $ \imath:X\rightarrow\mathfrak{P}_{w}(X) $ defined by $ \imath(x)=\{x\} $ is a diffeological immersion. In particular, $ \mathrm{dim}(X) \leq \mathrm{dim}(\mathfrak{P}_{w}(X)) $.
\end{Proposition}
\begin{proof} 
It is clear that $ \imath $ is smooth.
Let $ P:U\rightarrow \mathfrak{P}_{w}(X) $ be a plot and $ (r_0,x_0)\in P^*X $,
or $ P(r_0)=\imath(x_0)=\{x_0\} $. 
By \textbf{WPD}, for $ r_0\in U $, there exist an open neighborhood $ V \subset U $ of $ r_0  $ and a local selection plot $ \sigma : V \rightarrow X $  such that $ \sigma(r) \in P(r) $, 
for every $ r \in V $.
Let $ O:=P^*X\cap(V\times X) $, which is a D-open neighborhood of $ (r_0,x_0) $ in $ P^*X $, and define $ \rho:V\rightarrow U\times X $ by {$ \rho(r)=(r,\sigma(r)) $}. Then for each $ (r,x)\in O $, we have $ \sigma(r) \in P(r)=\imath(x)=\{x\} $, which yields
\begin{center}
$ \rho\circ P^*\imath(r,x)=\rho(r)=(r,\sigma(r))=(r,x) $.
\end{center}
Therefore,  $ \imath $ is a diffeological immersion.
\end{proof}

\begin{rem} 
Similar to \cite[Exercise 32]{Igdiff}, one can observe that for any equivalence relation $ R $ over $ X $, the inclusion
$ X/R \hookrightarrow \mathfrak{P}_{w}(X) $
is an induction, where  $ X/R $ is endowed with
the quotient diffeology. However, the  finest diffeology on $ \mathfrak{P}(X) $ with this property has the following description.
Let $ \mathsf{EqRel}(X) $ denote the set of all equivalence relations over $ X $.
Consider the sum space $ Y=(\bigsqcup_{R\in \mathsf{EqRel}(X)}X/R) \bigsqcup \{\emptyset\} $,  where  $ X/R $ is endowed with
the quotient diffeology.
In this way, one can define the push-forward diffeology $ \mathcal{D} $ on $ \mathfrak{P}(X) $ through the map $ Y\rightarrow\mathfrak{P}(X) $ taking $ (R,A)\in X/R \mapsto A \in \mathfrak{P}(X) $ and $ \emptyset\mapsto\emptyset $.
More explicitly, a parametrization $ P:U\rightarrow\mathfrak{P}(X) $ is an element of $ \mathcal{D} $ if and only if 
\begin{enumerate}
\item[ $ \bullet $]
for each $ r_0 \in U $, there
exists an open neighborhood $ V \subset U $ of $ r_0 $ such that either
$ P|_V $ is constant with the value $ \emptyset $,
or $ P|_V $ is a plot in $ \mathfrak{P}^{\star}_{w}(X) $ such that
$ P(r)\cap P(r')=\emptyset $ or $ P(r)= P(r')$, for all $ r\neq r'\in V $.
\end{enumerate} 	
In fact,  $ \mathcal{D} $ is the finest diffeology on $ \mathfrak{P}(X) $ with this feature that for any equivalence relation $ R $ over $ X $, the inclusion
$ X/R \hookrightarrow (\mathfrak{P}(X),\mathcal{D}) $
is an induction.
\end{rem}

From the viewpoint of usual topologies, these versions of the power set diffeologies may be too coarse, because it seems to make too many functions smooth in basic contexts as shown in the next example.

\begin{example}
The set-valued map $ \phi:\mathbb{R}\rightarrow \mathfrak{P}_{w}(\mathbb{R}) $ given by
\begin{align*}
\phi(x)= \left\{
\begin{array}{lr}
	\{0,1\}\quad & x\in\mathbb{Q}\\
	\{0\} \quad &   x\notin\mathbb{Q}\\
\end{array} \right.
\end{align*}
is smooth. This map is not lower semi-continuous, because for open subset
$ (\frac{1}{2},\frac{3}{2})\subset\mathbb{R} $,
one has $ \phi^{-1}((\frac{1}{2},\frac{3}{2}))=\{ x\in \mathbb{R} \mid \phi(x)\cap (\frac{1}{2},\frac{3}{2})\neq\emptyset\}=\mathbb{Q} $, which is not open. 
\end{example}	

To avoid such pathological behaviors of the weak power set diffeology, in the next section we introduce and investigate a refinement of this diffeology. 
\section{The union power set diffeology}\label{upsdiff}
In this section, we introduce the second type of diffeologies for the power sets and investigate its properties and behaviors in various aspects.
\begin{Definition}
Let  $ \lbrace P_i: U_i\rightarrow \mathfrak{P}(X)\rbrace_{i\in J} $  be a family of $ n $-parametrizations in $ \mathfrak{P}(X) $. The parametrization
$ \bigcup_{i\in J} P_i:\bigcup_{i\in J} U_i\rightarrow \mathfrak{P}(X) $ defined by
\begin{center}
$ (\bigcup_{i\in J} P_i)(r)=\bigcup\{ P_i(r)\mid i\in J, r\in \mathrm{dom}(P_i) \}, \quad $ for each $ r\in U $, 
\end{center}
is called the \textbf{union} of this family.
By convention, the union of the empty family is the parametrization $ \emptyset\rightarrow \mathfrak{P}(X) $.
\end{Definition}

One can understand the (not necessarily compatible) family $ \lbrace P_i: U_i\rightarrow \mathfrak{P}(X)\rbrace_{i\in J} $ of $ n $-parametrizations with the compatible family 
\begin{center}
$ \Pi_i:U_i\rightarrow \mathfrak{P}(X), \quad r\mapsto\bigcup\{ P_j(r)\mid j\in J, r\in \mathrm{dom}(P_j) \} $.
\end{center}
of $ n $-parametrizations such that the supremum $ \lbrace \Pi_i\rbrace_{i\in J} $ is equal to the union of $ \lbrace P_i\rbrace_{i\in J} $.

If $\lbrace P_i: U_i\rightarrow X\rbrace_{i\in J}$ is a compatible family of $ n $-parametrizations in $ X $ (regarded as those in $ \mathfrak{P}(X) $),  the union parametrization is equal to the supremum  of this family.

\begin{Definition}\label{d:union-diff} 
Let $ \A $ be a collection of parametrizations in $ \mathfrak{P}(X) $  such that
\begin{enumerate}
\item[(a)]  $ \A $ contains 0-parametrizations $ 0\mapsto \emptyset $ and $ 0\mapsto \{x\} $, for all $ x\in X $,

\item[(b)]  $ \A $ is stable under the smooth compatibility condition in Definition \ref{d:diffeology}, i.e.,
for every element $P:U\rightarrow X$ of $ \A $ and every smooth map $ F:V\rightarrow U $ between domains, the composition $ P\circ F$ belongs to $ \A $.
\end{enumerate}
The set of  parametrizations $ P $
that are the union of a family
$ \lbrace P_i\rbrace_{i\in J} $ of parametrizations in $ \mathfrak{P}(X) $ with  $ P_i\in{\A} $ is a diffeology on $ \mathfrak{P}(X) $, called the \textbf{union diffeology}  generated by $ \A $. 
\end{Definition}  
Let $ \A $ be an arbitrary collection of parametrizations in $ \mathfrak{P}(X) $. Extended $ \A $ to $ \overline{\A} $ by
adding  0-parametrizations $ 0\mapsto \emptyset $ and $ 0\mapsto \{x\} $, $ x\in X $, to $ \A $. Composing $ \overline{\A} $ with smooth map $ F:V\rightarrow U $ between domains yields a new collection satisfying the requirements (a) and (b) of Definition \ref{d:union-diff}, which generates a union diffeology.
In particular, a set-valued map $ \phi:X'\rightarrow \mathfrak{P}(X) $ defined on a diffeological space $ X' $ induces a diffeology on $ \mathfrak{P}(X) $ as the union diffeology generated by the collection $ \{\phi\circ P\mid P$ is a plot in  $X' \} $.

\begin{Definition}\label{d:upd}
Let  $ X $  be a diffeological space.  
The \textbf{union power set diffeology} on $ \mathfrak{P}(X) $ is the union diffeology generated by the canonical map $ \imath:X\rightarrow\mathfrak{P}(X)$ defined by $ \imath(x)= \{x\} $.
We denote by $  \mathfrak{P}_{u}(X) $  the power set	$  \mathfrak{P}(X) $ endowed with the union power set diffeology and call it the \textbf{union power space}.
\end{Definition} 	
\begin{Proposition}	
Let $ X $ be a diffeological space.
A parametrization 
$ P:U\rightarrow\mathfrak{P}(X) $ is a plot in $  \mathfrak{P}_{u}(X) $ if and only if 
\begin{enumerate}
\item[\textbf{UPD.}] 
for every $ r_0 \in U $ and any initial condition $ x_0\in P(r_0) $, there exist an open neighborhood $ V \subset U $ of $ r_0 $ and a  plot $ \sigma : V \rightarrow X $ in $ X $, called \textbf{local implicit plot}, such that $ \sigma(r) \in P(r) $,  
for every $ r \in V $ and $ \sigma(r_0)=x_0 $.
\end{enumerate}
\end{Proposition}
\begin{proof}
Suppose that $ P:U \rightarrow \mathfrak{P}_{u}(X) $ is a plot.
By definition, $ P $ is the union of a collection $ \{\imath\circ P_i\mid  P_i:U_i\rightarrow X$ is a plot in  $X, i\in J \} $ and probably, the  constant parametrization with the value $\emptyset $ defined on some open subset of $ U $.
Let $ r_0 \in U $ and $ x_0\in P(r_0)=\bigcup_{i\in J} \imath\circ P_i(r_0) $. 
Then 
$ x_0\in \imath\circ P_i(r_0) $ and $ r_0 \in U_i $, for some $ i\in J $. Hence 
$ P_i:U_i\rightarrow X $ is a local implicit plot such that $ P_i(r) \in P(r) $  
for every $ r \in U_i\subset U $, and $ P_i(r_0)=x_0 $.

Conversely, suppose that $ P:U\rightarrow\mathfrak{P}(X) $ is a parametrization with the property \textbf{UPD}.
Let $ \mathcal{I} $ be the collection of all implicit plots $ \sigma $ of $ P $ and the  constant parametrization defined on some open subset of $ U $ with the value $\emptyset $.
Let $ Q $ be the union parametrization of the family of $ \mathcal{I} $. 
We prove that $ Q=P $.
Firstly, the domain of definition of $ Q $ is exactly $ U $.
Now let $ r_0\in U $ be arbitrary. It is obvious that $ Q(r_0)\subset P(r_0) $.
So let $ x_0\in P(r_0) $. Then there exists an implicit plot $ \sigma : V \rightarrow X $, defined on an open neighborhood $ V \subset U $ of $ r_0 $ with $ \sigma(r_0)=x_0 $, which implies $ x_0\in Q(r_0) $.
Hence $ Q=P $. Indeed, $ P $ is the union of its local implicit plots and a plot in $ \mathfrak{P}_{u}(X) $.
\end{proof}	
\begin{Cor}\label{cor:weak-union}
The identity map 
$ \mathrm{id}:\mathfrak{P}_{u}(X)\rightarrow \mathfrak{P}_{w}(X) $
is smooth.
In other words, the union power set diffeology is finer than the weak power set diffeology.
\end{Cor} 


\subsection{Smooth set-valued mappings}
Let $ X $ and $ Y $ be diffeological spaces.
We can simply treat a smooth set-valued map $ \phi:X\rightarrow\mathfrak{P}_{u}(Y) $ 
as a usual smooth map, that is, a map taking the plots in $ X $ to the plots in $ \mathfrak{P}_{u}(Y) $.

\begin{example}
Suppose that $ X $ is a  diffeological space and $ f,g:X\rightarrow\mathbb{R}  $ are two smooth functions.
The map $ \phi:X\rightarrow\mathfrak{P}_{u}(\mathbb{R}) $ defined by  
$$ \phi(x)=[\min\{f(x),g(x)\}, \max\{f(x),g(x)\}] $$
is smooth. To verify this,
let $ P:U\rightarrow X $ be a plot and let $ r_0\in U $, $ a_0\in\phi\circ P(r_0) $.
Then we have two cases:

(i) If $ f(P(r_0))\neq g(P(r_0)) $: Let $ t_0=\dfrac{a_0-g\circ P(r_0)}{f\circ P(r_0)-g\circ P(r_0)} $, which is a number belonging to $ [0,1] $.  Define $ \sigma:U\rightarrow \mathbb{R} $ by $ \sigma(r)=(t_0)f\circ P (r)+(1-t_0)g\circ P(r) $. 
Since the sum and the product in $ \mathbb{R} $ are smooth, $ \sigma $ is also smooth. One can check that $ \sigma(r)\in \phi\circ P(r) $. In particular, $ \sigma(r_0)=(t_0)f\circ P (r_0)+(1-t_0)g\circ P(r_0)=a_0 $. 

(ii) If $ f(P(r_0))=g(P(r_0)) $: It is enough to choose $ \sigma=f\circ P $. Clearly, $ \sigma $ has the desired properties.
\end{example}	
\begin{Theorem}	\label{The-lsc}
Let $ X $ and $ Y $ be diffeological spaces.
Suppose that $ \phi:X\rightarrow\mathfrak{P}_{u}(Y) $ is a smooth set-valued map.
Then $ \phi $ is lower semi-continuous with respect to the D-topology.
\end{Theorem}
\begin{proof}
Assume that $ O $ is a D-open subset of $ Y $. Let $ P:U\rightarrow X $ be any plot and let $ r_0\in {P}^{-1}\big{(}\phi^{-1}(O)\big{)} $ be arbitrary, 
where $ \phi^{-1}(O)=\{ x\in X \mid \phi(x)\cap O\neq\emptyset\} $. 
So there is an element $ y_0 $ in $ \phi\circ P(r_0)\cap O $.
Because $ \phi $ is smooth, $ \phi\circ P $ is a plot in $ \mathfrak{P}_{u}(Y) $. There exist an open neighborhood $ V \subset U $ of $ r_0 $ and a plot $ \sigma : V \rightarrow Y $  such that $ \sigma(r) \in \phi\circ P(r) $, for every $ r \in V $, and $ \sigma(r_0)=y_0 $.
If $ r\in\sigma^{-1}(O) $, then $ \sigma(r) \in \phi\circ P(r)\cap O\neq\emptyset $ and hence 
$ r\in (\phi\circ P)^{-1}(O)={P}^{-1}\big{(}\phi^{-1}(O)\big{)} $, which implies
$ \sigma^{-1}(O)\subset {P}^{-1}\big{(}\phi^{-1}(O)\big{)}$.
But $ \sigma^{-1}(O) $ is an open neighborhood of $ r_0 $ so that $ {P}^{-1}\big{(}\phi^{-1}(O)\big{)}$ is open as well. Therefore $ \phi^{-1}(O) $ is D-open.
\end{proof}
\begin{Proposition}\label{pro:union-se}
The map $ \imath:X\rightarrow\mathfrak{P}_{u}(X) $ defined by $ \imath(x)=\{x\} $ is a strong embedding.
\end{Proposition}
\begin{proof}
It is easily seen that $ \imath $ is an induction.
Also, since the identity map 
$ \mathrm{id}:\mathfrak{P}_{u}(X)\rightarrow \mathfrak{P}_{w}(X) $
is smooth and $ \imath:X\rightarrow\mathfrak{P}_{w}(X) $ is a diffeological immersion, 
$ \imath:X\rightarrow\mathfrak{P}_{u}(X) $ is a diffeological immersion, by Proposition \ref{p:immr}.
We show
$ \imath $ is a D-topological embedding, too.
Assume that $ O\subset X $ is D-open. Set
$ O'=\{A\in\mathfrak{P}(X)\mid A\cap O\neq\emptyset\} $ so that
$ \imath(O)=\imath(X)\cap O' $.
It is sufficient to prove that $ O' $ is D-open.
Let $ P:U\rightarrow \mathfrak{P}_{u}(X) $ be a plot. 
By Theorem \ref{The-lsc},  
$ P $ is lower semi-continuous with respect to the D-topology. So 
\begin{center}
$ P^{-1}(O')=\{r\in U\mid P(r)\in O'\} =\{r\in U\mid P(r)\cap O\neq\emptyset\} $
\end{center}
is open, which means $ O' $
is D-open.
Therefore, it is a strong embedding.
\end{proof}

\begin{Cor} 
$ \mathrm{dim}(X) \leq \mathrm{dim}(\mathfrak{P}_{u}(X)) $, for any diffeological space $ X $.
\end{Cor}

\begin{example}

In the above statement, the equality may or may not be true.
As example, the equality holds for $ X=\mathbb{R}^0 $. 
While for the discrete space $ \mathbb{Q} $ of the rational numbers, the equality does not hold.
Because, the parametrization 
$ P:\mathbb{R}\rightarrow\mathfrak{P}_{u}(\mathbb{\mathbb{Q}}) $
given by $ P(0)=\{0\} $ and
$ P(r)=[-|r|,|r|]\cap \mathbb{Q} $ for $ r\neq 0 $, is a plot in $ \mathfrak{P}_{u}(\mathbb{Q})  $, which is not locally constant.
Therefore, $ \mathrm{dim}(\mathbb{Q})=0<\mathrm{dim}(\mathfrak{P}_{u}(\mathbb{Q})) $.

\end{example}


\begin{Proposition}
Let $ X $ and $ Y $ be diffeological spaces.
A set-valued map $ \phi:X\rightarrow\mathfrak{P}_{u}(Y) $ is smooth if and only if the extension $ \widetilde{\phi}:\mathfrak{P}_{u}(X)\rightarrow\mathfrak{P}_{u}(Y) $ defined by $ \widetilde{\phi}(A)=\bigcup_{x\in A} \phi(x) $ for every subset $ A\subset X $, is smooth.
\end{Proposition}
\begin{proof}
The proof is quite similar to that of \cite{ADD2020}. 
\end{proof}
\begin{Theorem}
The functional diffeology on $ C^{\infty}(X,\mathfrak{P}_{u}(Y)) $ is the coarsest diffeology for which
$ \widetilde{\mathrm{ev}}: C^{\infty}(X,\mathfrak{P}_{u}(Y))\times\mathfrak{P}_{u}(X)\rightarrow \mathfrak{P}_{u}(Y)  $ defined by $ \widetilde{\mathrm{ev}}(\phi,A)=\widetilde{\phi}(A) $ is smooth.
In particular, the operator $ \sim: C^{\infty}(X,\mathfrak{P}_{u}(Y))\rightarrow C^{\infty}(\mathfrak{P}_{u}(X),\mathfrak{P}_{u}(Y))  $ taking $  \phi $ to the extension $ \widetilde{\phi} $ is smooth.
\end{Theorem}
\begin{proof}
Assume that $ P:U\rightarrow C^{\infty}(X,\mathfrak{P}_{u}(Y)) $ is a plot with respect to the functional diffeology and
denote by $ \widetilde{P}:U\rightarrow C^{\infty}(\mathfrak{P}_{u}(X),\mathfrak{P}_{u}(Y)) $
the map given by $ \widetilde{P}(r)=\widetilde{P(r)} $.
Let $ Q:V\rightarrow\mathfrak{P}_{u}(X) $ be a plot in $ \mathfrak{P}_{u}(X) $,
$ (r_0,s_0) \in U\times V $ and $ y_0\in\widetilde{\mathrm{ev}}\circ (\widetilde{P}\times Q)(r_0,s_0) $. Then there exists $ x_0\in Q(s_0) $ such that $ y_0\in P(r_0)(x_0) $.
So we have a plot $ \sigma:V'\rightarrow X $ defined on an open neighborhood of $ s_0\in V'\subset V $ with $ \sigma(s_0)=x_0 $ and  $ \sigma(s)\in Q(s) $ for all $ s\in V' $.
On the other hand, 	$ \mathrm{ev}: C^{\infty}(X,\mathfrak{P}_{u}(Y))\times X\rightarrow \mathfrak{P}_{u}(Y)  $ is smooth. So 
$ \mathrm{ev}\circ(P\times \sigma):U\times V'\rightarrow\mathfrak{P}_{u}(Y) $ is a plot in $ \mathfrak{P}_{u}(Y) $ that  $ y_0\in \mathrm{ev}\circ(P\times \sigma)(r_0,s_0) $. Consequently, we obtain a plot $ \tau:W \rightarrow Y $ defined on an open neighborhood of $ (r_0,s_0)\in W\subset U\times V' $ with $ \tau(r_0,s_0)=y_0 $ and  $ \tau(r,s)\in \mathrm{ev}\circ(P\times \sigma)(r,s)\subset\widetilde{\mathrm{ev}}\circ (\widetilde{P},Q)(r,s) $ for all $ (r,s)\in W $.
As a result, the way of proof shows that
the map $ \sim: C^{\infty}(X,\mathfrak{P}_{u}(Y))\rightarrow C^{\infty}(\mathfrak{P}_{u}(X),\mathfrak{P}_{u}(Y))  $ is smooth as well.

Finally, if $ \mathcal{D} $ is any diffeology on $ C^{\infty}(X,\mathfrak{P}_{u}(Y)) $ for which
$ \widetilde{\mathrm{ev}}  $ is smooth, then obviously the composition
\begin{center}
$ \mathrm{ev}:C^{\infty}(X,\mathfrak{P}_{u}(Y))\times X\stackrel{\mathrm{id}\times \imath}{\longrightarrow} C^{\infty}(X,\mathfrak{P}_{u}(Y))\times\mathfrak{P}_{u}(X)\stackrel{\widetilde{\mathrm{ev}}}{\longrightarrow} \mathfrak{P}_{u}(Y) $
\end{center}
is smooth as well,
which is equal to
$ \mathrm{ev}: C^{\infty}(X,\mathfrak{P}_{u}(Y))\times X\rightarrow \mathfrak{P}_{u}(Y)  $.
Therefore, $ \mathcal{D} $ is finer than the functional diffeology.
\end{proof}

\begin{Proposition}
Let $ f:X\rightarrow Y $ be a map between diffeological spaces and $ f_\star:\mathfrak{P}_{u}(X) \rightarrow\mathfrak{P}_{u}(Y) $ be the induced hyper-map, taking every subset $ A\subset X $ to $ f(A) $.
\begin{enumerate}
\item[(1)]
The map $ f $ is smooth if and only if $ f_\star $ is smooth.
\item[(2)]
The map $ f $ is an induction if and only if $ f_\star $ is an induction.
\item[(3)]
The map $ f $ is a diffeomorphism if and only if $ f_\star $ is a diffeomorphism.
\item[(4)]
If $ f_\star $ is a diffeological immersion, $ f $ is a diffeological immersion.
\item[(5)]
If $ f_\star $ is a diffeological submersion, $ f $ is a diffeological submersion.
\item[(6)]
If $ f_\star $ is a diffeological \'etale map, $ f $ is a diffeological \'etale map.
\item[(7)]
If $ f_\star $ is a local subduction, $ f $ is a local subduction.
\item[(8)]
If $ f_\star $ is a subduction, $ f $ is a subduction.
\item[(9)]
If  $ f $ is a local subduction, then $ f_\star $ is a subduction.
\end{enumerate} 

\end{Proposition}
\begin{proof}
(1)	Assume that $ f $ is smooth and $ P:U\rightarrow\mathfrak{P}_{u}(X) $ is a plot. To show $ f_\star\circ P:U\rightarrow\mathfrak{P}_{u}(Y)  $  is a plot in $ \mathfrak{P}_{u}(Y) $, let $ r_0 \in U $, $ y_0\in f_\star\circ P(r_0) $.
Then there is an element $ x_0\in P(r_0) $ with $ y_0=f(x_0) $.
By \textbf{UPD}, there exist an open neighborhood $ V \subset U $ of $ r_0 $ and a plot $ \sigma: V \rightarrow X $ such that $ \sigma(r_0)=x_0 $ and $ \sigma(r) \in P(r) $, 
for every $ r \in V $. But $ f\circ \sigma $ is a plot in $ Y $ with $ f\circ \sigma(r)=f(x_0)=y_0 $ and one can write
$ f\circ \sigma(r) \in f(P(r))=f_\star\circ P(r) $, for every $ r \in V $.
Thus, the parametrization $ f_\star\circ P:U\rightarrow\mathfrak{P}_{u}(Y)  $  is a plot in $ \mathfrak{P}_{u}(Y) $ and  $ f_\star $ is smooth.
Conversely, if $ f_\star $ is smooth, then the composition $ f_\star\circ\imath $ is also smooth.
Hence, $ f $ is smooth.

(2)
Suppose first that  $ f $ is an induction.
To see that $ f_\star $ is  injective, let $ f(A)=f(B) $ for subsets $ A $ and $ B $ of $ X $.
Let $ a\in A $, then $ f(a)\in f(A)=f(B) $. There is some $ b\in B $ such that $  f(a)=f(b) $.
Thus, $ a=b $ and $ a $ belongs to $ B $, that is, $ A\subset B $. Interchanging the roles of elements of $ A $ and $ B $, we obtain the equality  $ A= B $. 
The smoothness of $ f_\star $ is a consequence of (1).
So assume that  $ P:U\rightarrow\mathfrak{P}_{u}(X) $ is a parametrization such that $ f_\star\circ P $ is a plot in $ \mathfrak{P}_{u}(Y) $.
To prove that $ P $ is a plot in $ \mathfrak{P}_{u}(X) $, let $ r_0 \in U $, $ x_0 \in P(r_0) $.
By \textbf{UPD}, there exist an open neighborhood $ V \subset U $ of $ r_0 $ and a plot $ \sigma: V \rightarrow Y $ such that $ \sigma(r_0)=f(x_0) $ and $ \sigma(r) \in f_\star\circ P(r) $, 
for every $ r \in V $.
Thus, for every $ r \in V $, one can choose an element $ \tau_r\in  P(r) $  with $ \sigma(r)=f(\tau_r) $.
In particular, $ \tau_{r_0}=x_0 $.
Define the parametrization $\tau:V\rightarrow X $ by $ r\mapsto \tau_r $ so that $ f\circ\tau=\sigma $.
Since $ f $ is an induction, we conclude that $ \tau $ is a plot in $ X $ and hence $ P $ is a plot in $ \mathfrak{P}_{u}(X) $.
Therefore, $ f_\star $ is an induction.

Conversely, let $ f_\star $ be an induction. Obviously, $ f $ is  injective.
Assume that $ P:U\rightarrow X $ is a parametrization such that $ f\circ P $ is a plot in $ Y $.
Then 
$ \imath\circ f\circ P $, or equivalently, $ f_\star\circ\imath\circ P $ is a plot in $ \mathfrak{P}_{u}(Y) $.
Since $ f_\star\circ\imath $ is an induction, $ P $ is a plot in $ X $.

(3)	By item  (2) and this fact that $ f $ is surjective if and only if $ f_\star $ is surjective, the result is achieved.

(4) Smoothness of $ f $ is a consequence of (1). We also have the commutative diagram
\begin{displaymath}
\xymatrix{
	\mathfrak{P}_{u}(X)\ar[r]^{f_\star}	&  \mathfrak{P}_{u}(Y)   \\
	X\ar[u]^{\imath}\ar[r]_{f} & Y\ar[u]_{\imath}  \\
}
\end{displaymath}
If $ f_\star $ is a diffeological immersion, then so is $ f_\star\circ\imath $. Consequently, $ f $ is a diffeological immersion by Proposition \ref{p:immr}.

(5)
Smoothness of $ f $ is a consequence of (1).
Assume that $ P:U\rightarrow Y $ is a plot in $ Y $ and let $ r_0\in U $, $ x_0\in f^{-1}(P(r_0)) $.
Then $ \imath\circ P $ is a plot in $ \mathfrak{P}_{u}(Y) $. 
As  $ f_\star $ is a diffeological submersion, there exists a plot $ L:V\rightarrow  \mathfrak{P}_{u}(X) $, where $ V \subset U $ is an open neighborhood of $ r_0 $, such that $ f_\star\circ L=\imath\circ P|_V $ and that $ P(r_0)=\{x_0\} $.
By \textbf{UPD}, we get a plot $ \sigma:W\rightarrow X $, where $ W \subset V $ is an open neighborhood of $ r_0 $, such that $ \sigma(r_0)=x_0 $ and
$ \sigma(r)\in L(r) $, for every $ r \in W $.
Obviously, $ f\circ\sigma(r) \in f_\star\circ L(r)=\imath\circ P(r)=\{P(r)\} $ for every $ r \in W $,  and so $ f\circ\sigma=P|_W $.

(6)
This is a direct consequence of items (4) and (5).

(7)
This is clear by item  (5) and this fact that $ f $ is surjective if and only if $ f_\star $ is surjective.

(8) The proof is similar to (5).

(9) 
Again, the smoothness of $ f_\star $ is a consequence of (1). 
Also,	$ f $ is surjective if and only if $ f_\star $ is surjective. 
Now consider a plot $ Q:U\rightarrow\mathfrak{P}_{u}(Y) $.
Define the parametrization $ P:U\rightarrow \mathfrak{P}_{u}(X) $ by $ P(r):=f^{-1}(Q(r)) $.
To verify $ P $ is a plot, let $ r_0\in U $ and $ x_0\in P(r_0) $ so that $ f(x_0)\in Q(r_0) $.
By PD, there is a plot $ \tau:V\rightarrow Y $, where $ V \subset U $ is an open neighborhood of $ r_0 $, such that  $ \tau(r_0)=f(x_0) $ and $ \tau(r) \in Q(r) $, for every $ r \in V $.
Because $ f $ is a local subduction, one can find a plot $ \sigma:W\rightarrow X $, where $ W\subset V $ is an open neighborhood of $ r_0 $, such that $ \sigma(r_0)=x_0 $ and $ f\circ\sigma=\tau $, which implies $ \sigma(r) \in f^{-1}(\tau(r))\subset f^{-1}(Q(r))=P(r)$, for every $ r \in W $.
Thus, $ P $ is a plot in $ \mathfrak{P}_{u}(X) $.
Finally, we have
$ f_\star\circ P(r)=f_\star(f^{-1}(Q(r))=Q(r) $.
Therefore,  $ f_\star $ is a subduction.
\end{proof}

Denote by $\mathsf{Diff}$ the category of diffeological spaces and smooth maps.

\begin{Cor}
The correspondence  $ \mathfrak{P}_{u}:\mathsf{Diff}\rightarrow\mathsf{Diff}  $ given by
$ X\mapsto\mathfrak{P}_{u}(X) $ in the level of objects
and $ f\mapsto f_\star $ in the level of morphisms defines a faithful functor.
\end{Cor}

\subsection{Higher order union power spaces}
\begin{notation}
Let $ X $ be a diffeological space.
Denote $ \mathfrak{P}_{u}(\mathfrak{P}_{u}(X)) $ by $ \mathfrak{P}^2_{u}(X) $ and inductively, 
$ \mathfrak{P}_{u}(\mathfrak{P}^{n-1}_{u}(X)) $ by $ \mathfrak{P}^n_{u}(X) $, for every integer $ n\geq 2 $.
\end{notation}

\begin{Proposition}
Let $ X $ be a diffeological space.
The map
$ \tau_X:\mathfrak{P}_{u}(X)\rightarrow \mathfrak{P}^2_{u}(X) $ taking any subset $ A\subset X $ to the subspace topology on $ A $ induced from the D-topology of $ X $, i.e.
$ \tau_X(A)=\{A\cap O\mid O\subset X $ is D-open $ \} $,
is a strong embedding.
\end{Proposition}
\begin{proof}
To check that $ \tau_X $ is smooth, 
let  $ P:U\rightarrow \mathfrak{P}_{u}(X) $  be a plot and $ r_0\in U, O_0\in\tau_X\circ P(r_0) $.
Then $ O_0= P(r_0)\cap O'_0 $, for some D-open subset $ O'_0\subset X $.
Define $ Q:U\rightarrow \mathfrak{P}_{u}(X) $ by $ Q(r)=P(r)\cap O'_0 $ for all $ r\in U $ so that $ Q(r)\in \tau_X\circ P(r) $ and $ Q(r_0)=O_0 $. If we prove that $ Q $ is a plot in $ \mathfrak{P}_{u}(X) $, then $ \tau_X\circ P $ is a plot in $ \mathfrak{P}^2_{u}(X) $ by definition. 
So let $ s_0\in U $ and $ x_0\in Q(s_0)=P(s_0)\cap O'_0  $.
As $ x_0\in P(s_0) $ and $ P $ is a plot in $ \mathfrak{P}_{u}(X) $, we deduce that  there exists an implicit plot $ \sigma : V \rightarrow X $, defined on an open neighborhood $ V \subset U $ of $ s_0 $ with $ \sigma(s_0)=x_0 $ and $ \sigma(r)\in P(r) $ for all $ r\in V $. 
Also, since $ x_0\in O'_0 $ and $ O'_0 \subset X $ is a D-open subset, after shrinking $ V $, we obtain $ \sigma(r)\in P(r)\cap O'_0 =Q(r)  $ on some neighborhood of $ s_0 $. Thus, $ Q $ is a plot in $ \mathfrak{P}_{u}(X) $.

On the other hand, consider the map $ \theta:\mathfrak{P}^2_{u}(X)\rightarrow \mathfrak{P}_{u}(X) $ taking any collection $ \mathcal{C} $ of the subsets of $ X $ to its union, i.e. $ \theta(\mathcal{C})=\bigcup \mathcal{C} $.
To observe that $ \theta $ is smooth, let $ \Pi:U\rightarrow \mathfrak{P}^2_{u}(X) $ be a plot and $ r_0\in U, x_0\in\theta\circ \Pi(r_0)=\bigcup \Pi(r_0) $.
Then $ x_0\in A_0 $ for some $ A_0\in \Pi(r_0) $, and by \textbf{UPD}, there exists a plot $ Q : V \rightarrow \mathfrak{P}_{u}(X) $, defined on an open neighborhood $ V \subset U $ of $ r_0 $ such that $ Q(r_0)=A_0\ni x_0 $ and $ Q(r)\in\Pi(r) $. 
Again by \textbf{UPD}, one can find an implicit plot $ \sigma : W \rightarrow X $ with $ \sigma(r_0)=x_0 $ and $ \sigma(r)\in Q(r)\subset\bigcup \Pi(r)=\theta\circ \Pi(r) $. Hence $ \theta\circ \Pi $ is a plot in $ \mathfrak{P}_{u}(X) $.
Finally, $ \theta\circ\tau_X(A)=A $ for all $ A\in \mathfrak{P}_{u}(X) $, so that $ \tau_X $ is a strong embedding.
\end{proof}

As a result, for any diffeological space $ X $, one has two infinite sequences
\begin{center}
$ X\stackrel{\imath_X}{\longrightarrow}\mathfrak{P}_{u}(X)\stackrel{\imath_{\mathfrak{P}_{u}(X)}}{\longrightarrow}\mathfrak{P}^2_{u}(X)\longrightarrow\cdots\stackrel{}{\longrightarrow}\mathfrak{P}^n_{u}(X)\stackrel{\imath_{\mathfrak{P}^{n}_{u}(X)}}{\longrightarrow}\mathfrak{P}^{n+1}_{u}(X)\longrightarrow\cdots$
\end{center}
and
\begin{center}
$ \mathfrak{P}_{u}(X)\stackrel{\tau_X}{\longrightarrow}\mathfrak{P}^2_{u}(X)\longrightarrow\cdots\stackrel{}{\longrightarrow}\mathfrak{P}^n_{u}(X)\stackrel{\tau_{\mathfrak{P}^{n-1}_{u}(X)}}{\longrightarrow}\mathfrak{P}^{n+1}_{u}(X)\longrightarrow\cdots$
\end{center}
of consecutive strong embeddings.

\subsection{Smooth selection problem}

Here we state  smooth selection problem in the context of diffeology and explore its connections with smoothness of set-valued maps.
\begin{Definition}
Let  $ X $ and $ Y $ be diffeological spaces. 
Let $ \phi:X \rightarrow\mathfrak{P}^\star(Y) $ be a set-valued map. 
The \textbf{smooth selection problem} of $ \phi $ for given $ x_0\in X $ and $ y_0\in \phi(x_0) $ is finding a \textbf{smooth selection}, i.e., a smooth map $ \sigma:X\rightarrow Y $ such that $ \sigma(x)\in\phi(x) $ for all $ x\in X $, and $ \sigma(x_0)=y_0 $.
Furthermore, we call a smooth map $ \sigma:X\rightarrow Y $ with $ \sigma(x)\in\phi(x) $ for all $ x\in X $, a \textbf{weak smooth selection}.
\end{Definition}

\begin{example}
Let $ \phi:\mathbb{R}\rightarrow\mathfrak{P}^\star(\mathbb{\mathbb{R}}) $ be
given by $ \phi(0)=\{0\} $ and
$ \phi(r)=(-|r|,|r|) $ for $ r\neq 0 $. To solve the smooth selection problem, for every $ x_0\in \mathbb{R} $ and $ y_0\in \phi(x_0) $, consider the straight line $ \sigma:r\mapsto (\dfrac{y_0}{x_0})r $, passing through $ y_0 $ and $ 0 $, as a smooth selection of $ \phi $.

\end{example}

\begin{Lemma}\label{l:ssp}
Suppose that $ \phi:X \rightarrow\mathfrak{P}^\star(Y) $ is a set-valued map and $ f:X'\rightarrow X $ is a smooth map.
If for $ \phi:X \rightarrow\mathfrak{P}^\star(Y) $, the smooth selection problem has solutions, then so is $ \phi\circ f $.
\end{Lemma}
\begin{proof}
The proof is straightforward.
\end{proof}	
\begin{Proposition}
Let $ X $ and $ Y $ be diffeological spaces, and let $ \phi:X \rightarrow\mathfrak{P}^\star(Y) $ be a set-valued map.
$ \phi $ is smooth if and only if there exists some covering generating family $ \mathcal{C} $ of $ X $ such that for all $ P\in\mathcal{C} $, the smooth selection problem has solutions for $ \phi\circ P $.
\end{Proposition}
\begin{proof}
Assume that $ \phi:X \rightarrow\mathfrak{P}^\star(Y) $ is smooth and let $ P:U\rightarrow X $ be any plot in $ X $. 
Then $ \phi\circ P $ is a plot in $ \mathfrak{P}^\star(Y) $, and by \textbf{UPD}, there is an open cover $ \{U_i\}_{i\in J} $ of $ U $ such that the smooth selection problem has solutions for $ \phi\circ P|_{U_i} $.
Notice that the collection of such restrictions for all plots in $ X $ constitutes a covering generating family for $ X $.
The converse is trivial in view of Lemma \ref{l:ssp}.
\end{proof}	

Similarly, the smoothness of a set-valued map $ \phi:X \rightarrow\mathfrak{P}_w^\star(Y) $ is equivalent to solving the weak smooth selection problem over a covering generating family of $ X $. 
\begin{Proposition}\label{p:ssp}
Let $ X $ and $ Y $ be diffeological spaces, and let $ \phi:X \rightarrow\mathfrak{P}^\star(Y) $ be a set-valued map.
If for each $ x_0\in X $ and $ y_0\in\phi(x_0) $, there exist a D-open neighborhood $ O\subset X $ of $ x_0 $ and a smooth map $ \alpha:O\rightarrow Y $ such that $ \alpha(x_0)=y_0 $ and $ \alpha(x)\in \phi(x) $ for all $ x\in O $, then $ \phi:X \rightarrow\mathfrak{P}_u^\star(Y) $ is smooth with respect to the union power set diffeology.
\end{Proposition}
\begin{proof}
Let $ P:U\rightarrow X $ be a plot, let $ r_0\in U $ and $ y_0\in\phi\circ P(r_0) $. By hypothesis, there exist a D-open subset $ P(r_0)\in O $ of $ X $ and a smooth map $ \alpha:O\rightarrow Y $ such that $ \alpha(P(r_0))=y_0 $ and $ \alpha(x)\in \phi(x) $ for all $ x\in O $. Set $ V=P^{-1}(O) $, which is an open neighborhood of $ r_0 $ in $ U $. Then $ \alpha\circ P|_V $ is a plot in $ Y $ with $ \alpha\circ P|_V (r_0)=y_0 $ and $ \alpha\circ P|_V (r)\in \phi\circ P(r) $ for all $ r\in V $.
\end{proof}	
In particular, if one can solve smooth selection problem for a set-valued map $ \phi:X \rightarrow\mathfrak{P}^\star(Y) $,
then it is smooth with respect to the union power set diffeology.
Thus, smoothness is a necessary, but not a sufficient, condition for smooth selection problem.

For manifolds, however, the smoothness is equivalent to solving the smooth selection problem over a D-open cover of $ X $.

\begin{Proposition}\label{p:svlman}
Let $ M $ be a manifold, $ X $ be a diffeological space and let $ \phi:M \rightarrow\mathfrak{P}^\star(X) $ be a set-valued map.
Then 
$ \phi $ is smooth with respect to the union power set diffeology if and only if
for each $ p_0\in M $ and $ x_0\in\phi(p_0) $, there exist a D-open neighborhood $  O\subset M $ of $ p_0 $ and smooth map $ \alpha:O\rightarrow X $ such that $ \alpha(p_0)=x_0 $ and $ \alpha(p)\in \phi(p) $ for all $ p\in O $.
\end{Proposition}
\begin{proof}
The ``if" direction is clear by Proposition \ref{p:ssp}.
For ``only if", assume that $ \phi:M \rightarrow\mathfrak{P}_u^\star(X) $ is smooth, and let $ p_0\in M $ and $ x_0\in\phi(p_0) $.
Choose a chart $ \psi:U\rightarrow U'\subset M $ at $ p_0 $ and let $ p_0=\psi(r_0) $, for some $ r_0\in U $. Since $ \phi\circ\psi $ is a plot in $ \mathfrak{P}_u^\star(X) $,
there exist an open neighborhood $ V\subset U $ of $ r_0 $ and a plot $ \sigma:V\rightarrow Y $ such that $ \sigma(r_0)=x_0 $ and $ \sigma(r)\in \phi\circ\psi(r) $ for all $ r\in V $.
Now let $ O=\psi(V)\subset U' $ and $ \alpha=\sigma\circ \psi^{-1}|_{O} $. Then $ \alpha(p_0)=\sigma\circ \psi^{-1}(p_0)=\sigma(r_0)=x_0 $ and
$ \alpha(p)=\sigma\circ \psi^{-1}(p)\in\phi\circ\psi(\psi^{-1}(p))=\phi(p)  $, for all $ p\in O $.
\end{proof}	 
As a result, we observe that the notion of a smooth set-valued map is compatible with the classical one between Euclidean spaces (see, e.g., \cite[Theorem 2.2.1]{SZ}).

\subsection{Smooth relations and power set diffeologies}

It is well-known that there is a natural one-to-one correspondence between set-valued maps and relations.
In fact, if  $ \phi:X\rightarrow\mathfrak{P}(Y) $  is
a set-valued map, 
the subset 
$ Graph(\phi) =\{(x,y)\mid y\in \phi(x)\}$
of $ X\times Y $ is its corresponding relation.
On the other hand, a relation $ R $  from  $ X $ to  $ Y $ define a set-valued map $ \phi:X\rightarrow\mathfrak{P}(Y) $  by $ \phi(x)=\{y\in Y\mid (x,y)\in R \} $, and we get  $  Graph(\phi)=R $.
According to this one-to-one correspondence, a smooth relation between diffeological spaces makes sense.

\begin{Proposition}
Let  $ X $ and $ Y $ be diffeological spaces. 
Suppose that $ \phi:X\rightarrow\mathfrak{P}(Y) $ is a set-valued map and $  Graph(\phi) $ is its corresponding relation, as a subspace of $ X\times Y $. Then
\begin{enumerate}
\item[(i)] 
$ \phi:X\rightarrow\mathfrak{P}_{w}(Y) $ is smooth if and only if the map $ \Pr_1: Graph(\phi)\rightarrow X $ is a weak subduction.  
\item[(ii)] 
$ \phi:X\rightarrow\mathfrak{P}_{u}(Y) $ is smooth if and only if the map $ \Pr_1: Graph(\phi)\rightarrow X $ is a diffeological submersion.  	
\end{enumerate} 	

\end{Proposition}
\begin{proof}
The proof is straightforward.
\end{proof} 
This suggests to define smooth relation as follows  (see also \cite[Exercise 62(2)]{Igdiff}):
\begin{Definition}
Let $  R $ be a relation from a diffeological space $ X $ to a diffeological space $ Y $.
Consider $  R $ as a subspace of $ X\times Y $.
\begin{enumerate}
\item[(i)] 
$  R $  is smooth with respect to the weak power set diffeology if and only if $ \Pr_1:R\rightarrow X $ is a weak subduction.  
\item[(ii)] 
$  R $  is smooth with respect to the union power set diffeology if and only if $ \Pr_1:R\rightarrow X $ is a diffeological submersion.  
\end{enumerate} 	
We denote by $ \mathrm{Rel}^{\infty}(X,Y) $
the space of smooth (partially defined) relations from $ X $ to $ Y $
as a subspace of $ \mathfrak{P}_{u}(X\times Y) $. 
\end{Definition}

\begin{example}
Any equivalence relation $ R $ on a diffeological space $ X $ is smooth with respect to the weak power set diffeology.
\end{example}
\begin{Proposition}
Let $ X $ and $ Y $ be diffeological spaces. 
The set-valued map
$ Graph:C^{\infty}(X,\mathfrak{P}_{u}(Y))\rightarrow \mathrm{Rel}^{\infty}(X,Y) $ taking
$ f \mapsto Graph(f) $
is smooth.
In other words, the functional diffeology on $ C^{\infty}(X,\mathfrak{P}_{u}(Y)) $ is finer than
the subspace diffeology on it, inherited from the union power space $  \mathfrak{P}_{u}(X\times Y)  $.
\end{Proposition}

\begin{proof}
Assume that $ P:U\rightarrow C^{\infty}(X,\mathfrak{P}_{u}(Y)) $ is a plot with respect to the functional diffeology and
let $ r_0 \in U $ and $ (x_0,y_0)\in Graph\circ P(r_0) $, which means $ y_0\in P(r_0)(x_0) $.
For the constant plot $ c_{x_0}:U\rightarrow X $ with the value $ x_0 $, the map
$ ev(P, c_{x_0}) $ is a plot in $ \mathfrak{P}_{u}(Y) $.
So there is a plot $ \sigma:V\rightarrow Y $, where $ V \subset U $ is an open neighborhood of $ r_0 $, such that $ \sigma(r_0)=y_0 $ and
$ \sigma(r)\in ev(P, c_{x_0})(r)=P(r)(x_0) $, for every $ r \in V $.
Now consider the plot $ \tau=(c_{x_0}|_{V},\sigma) $ in $ X\times Y $ defined on $ V $. Obviously,  $ \tau(r)=(x_0, \sigma(r)) \in Graph\circ P(r) $ for every $ r\in V $, and that $ \tau(r_0)=(x_0, \sigma(r_0))=(x_0,y_0) $.
Hence $ Graph\circ P $ is a plot in $  \mathrm{Rel}^{\infty}(X,Y)  $ with respect to the union power set diffeology.
\end{proof}

\begin{Proposition}
The maps $ \mathrm{def}: \mathrm{Rel}^{\infty}(X,Y)\rightarrow \mathfrak{P}_{u}(X)  $ and $ \mathrm{Im}: \mathrm{Rel}^{\infty}(X,Y)\rightarrow \mathfrak{P}_{u}(Y)  $ are smooth.
\end{Proposition}
\begin{proof}
Let   $ P:U\rightarrow \mathrm{Rel}^{\infty}(X,Y) $ be a plot and let $r_0\in U, x_0\in \mathrm{def}\circ P(r_0) $.
Then there exists an element $ y_0\in Y $ such that $ y_0\in P(r_0)(x_0) \neq\emptyset $, or $ (x_0,y_0)\in P(r_0) $. There are an open neighborhood  $ V \subset U $ of $ r_0 $ and a plot $  \sigma:V\rightarrow X\times Y $ with $ \sigma(r)\in P(r) $ and $ \sigma(r_0)=(x_0,y_0) $. Then
$ \Pr_1\circ\sigma:V\rightarrow X $ is a plot such that $ \Pr_1\circ\sigma(r)\in \mathrm{def}\circ P(r) $ and $  \Pr_1\circ\sigma(r_0)=x_0 $. 

Just as the above argument, one can see that $ \mathrm{Im}: \mathrm{Rel}^{\infty}(X,Y)\rightarrow \mathfrak{P}_{u}(Y)  $ is smooth.
\end{proof}

\begin{Cor}
If $ C^{\infty}(X,\mathfrak{P}_{u}(Y)) $ is endowed with the functional diffeology, then
$ \mathrm{def}: C^{\infty}(X,\mathfrak{P}_{u}(Y))\rightarrow \mathfrak{P}_{u}(X)  $ and $ \mathrm{Im}: C^{\infty}(X,\mathfrak{P}_{u}(Y))\rightarrow \mathfrak{P}_{u}(Y)  $ are smooth.
\end{Cor}

\subsection{Smooth family of plots and the union power set diffeology}

{{} We here investigate the relationship between the power set diffeology and the diffeology on plots as it is defined in \cite{Igdiff}. For this, we first recall:}

\begin{Definition}
\cite[\S 1.63]{Igdiff} 
Let  $ (X, \mathcal{D}) $ be a diffeological space. The diffeological structure $ \mathcal{D} $ itself has a natural diffeology called the \textbf{standard
functional diffeology}. 
A parametrization $ \rho:U\rightarrow\mathcal{D} $ is a \textbf{plot} in $ \mathcal{D} $ or a \textbf{smooth family of plots in} $ X $ if and only if 
for all $ r_0\in U $, for all $ s_0\in\mathrm{dom}(\rho(r_0)) $, there exist an open neighborhood $  V\subset U $ of $ r_0 $
and an open neighborhood $  W $ of $ s_0 $ such that $ W \subset\mathrm{dom}(\rho(r)) $ for all $ r\in V $, and $ (r,s)\mapsto \rho(r)(s)$ defined on $ V\times W $ is a plot in $ X $.
\end{Definition}

\begin{Proposition}
The map  $ Im:\mathcal{D}\rightarrow \mathfrak{P}_{u}(X),~~ P\mapsto\mathrm{Im}(P)  $ is smooth. 
\end{Proposition}
\begin{proof}
Let  $ \rho:U\rightarrow\mathcal{D} $ be a plot in $ \mathcal{D} $, and  $ r_0\in U $,  $ x_0\in \mathrm{Im}(\rho(r_0)) $.
Thus, there is an $ s_0\in\mathrm{dom}(\rho(r_0)) $ such that $ \rho(r_0)(s_0)=x_0 $.
By definition, there exist an open neighborhood $  V\subset U $ of $ r_0 $
and an open neighborhood $  W $ of $ s_0 $ such that $ W \subset\mathrm{dom}(\rho(r)) $ for all $ r\in V $, and $ (r,s)\mapsto \rho(r)(s)$ defined on $ V\times W $ is a plot in $ X $.
Define $ \sigma:V\rightarrow X $ by $ \sigma(r)=\rho(r)(s_0)\in \mathrm{Im}(\rho(r)) $, as the desired local implicit plot in $ X $.
\end{proof}

Let $ \p_n(X) $ denote the set of $ n $-plots in $ X $ with the subspace diffeology inherited from $ \mathcal{D} $.
Consider the set-valued map $ dom:\p_n(X)\rightarrow \mathfrak{P}(\mathbb{R}^n) $ taking any $ n $-plot to its domain.

\begin{Proposition}
The map  $ dom:\p_n(X)\rightarrow \mathfrak{P}_u(\mathbb{R}^n) $ is smooth. 
\end{Proposition}
\begin{proof}
Let  $ \rho:U\rightarrow\p_n(X) $ be a smooth family of $ n $-plots in $ X $. Let $ r_0\in U $ and $ s_0\in\mathrm{dom}\circ\rho(r_0) $.
By definition, there exist an open neighborhood $  V\subset U $ of $ r_0 $
and an open neighborhood $  W $ of $ s_0 $ such that $ W \subset\mathrm{dom}(\rho(r)) $ for all $ r\in V $, and $ (r,s)\mapsto \rho(r)(s)$ defined on $ V\times W $ is a plot in $ X $.
Let $ \sigma:V\rightarrow W $ be the constant map with the value $ s_0 $, so that $ \sigma(r)\in W \subset\mathrm{dom}\circ\rho(r) $ for all $ r \in V $ and $ \sigma(r_0)=s_0 $.
\end{proof}
Consider
$ \Gamma_n:=\{(P,r)\mid P\in\p_n(X), r\in \mathrm{dom}(P) \} $ as a subspace of $ \p_n(X)\times\mathbb{R}^n  $. 
\begin{Proposition}
The map $ \mathrm{ev}:\Gamma_n\rightarrow X $ defined by $ \mathrm{ev}(P,r)=P(r) $  is smooth. 
\end{Proposition}
\begin{proof}
Assume that $ (\rho,F):U\rightarrow \Gamma_n $ is a plot. That is, $ \rho:U\rightarrow\p_n(X) $ is a smooth family of $ n $-plots in $ X $ and $ F:U\rightarrow \mathbb{R}^n  $ is a smooth map between domains with $ F(r)\in\mathrm{dom}(\rho(r)) $ for all $ r\in U $.
Take any $ r_0\in U $, so that $ F(r_0)\in\mathrm{dom}(\rho(r_0)) $.
By definition, there exist an open neighborhood $  V\subset U $ of $ r_0 $
and an open neighborhood $  W $ of $ F(r_0) $ such that $ W \subset\mathrm{dom}(\rho(r)) $ for all $ r\in V $, and $ (r,s)\mapsto \rho(r)(s)$ defined on $ V\times W $ is a plot in $ X $. Set $ V':=V\cap F^{-1}(W) $, which is an open neighborhood of $ r_0 $.
Thus, the parametrization $ ev\circ (\rho,F)|_{V'} $ taking $ r\in V' $ to $ \rho(r)(F(r)) $ is a plot in $ X $. Hence $ ev\circ (\rho,F) $ itself is a plot in $ X $.
\end{proof}
\begin{Proposition}
Consider $ \Delta=\{(P,F)\mid P\in\p_n(X), F\in\p_m(\mathbb{R}^n), \mathrm{Im}(F)\subset \mathrm{dom}(P) \} $  as a subspace of $ \mathcal{D}_n(X)\times\p_m(\mathbb{R}^n) $.
The map $ \circ:\Delta\rightarrow \p_m(X) $ taking any combinable pair $ (P,F) $ to $ P\circ F $ is smooth.
\end{Proposition}
\begin{proof}
Assume that $ (\rho,\nu):U\rightarrow\Delta $ be a plot. This means that $ \mathrm{Im}(\nu(r))\subset \mathrm{dom}(\rho(r)) $ for all $ r\in U $. To see the parametrization $ r\mapsto \rho(r)\circ\nu(r) $ on $ U $ is a smooth family of $ m $-plots in $ X $, let $ r_0\in U $ and  $ s_0\in\mathrm{dom}(\rho(r_0)\circ\nu(r_0))=\mathrm{dom}(\nu(r_0)) $.
By definition, there exist an open neighborhood $  V\subset U $ of $ r_0 $
and an open neighborhood $  W $ of $ s_0 $ such that $ W \subset\mathrm{dom}(\nu(r)) $ for all $ r\in V $, and $ (r,s)\mapsto \nu(r)(s)$ defined on $ V\times W $ is a plot in $ \mathbb{R}^n $.
On the other hand, for $ \nu(r_0)(s_0)\in\mathrm{dom}(\rho(r_0)) $, there exist an open neighborhood $  V'\subset V $ of $ r_0 $
and an open neighborhood $  W' $ of $ \nu(r_0)(s_0) $ such that $ W' \subset\mathrm{dom}(\rho(r)) $ for all $ r\in V' $, and $ (r,s)\mapsto \rho(r)(s)$ defined on $ V'\times W' $ is a plot in $ X $.
But there exists some open subsets $ W''\subset W $ and $ V''\subset V $ such that  $ (r,s)\mapsto \nu(r)(s)$ defined on $ V''\times W'' $ is a plot in $ W $.
Since the composition map
$ \circ: C^{\infty}(W',X)\times C^{\infty}(W,W') \rightarrow C^{\infty}(W,X)$
is smooth by \cite[\S 1.59]{Igdiff}, we conclude that  
$ (r,s)\mapsto (\rho(r)\circ\nu(r))(s)$ defined on $ V'\times W $ is a plot in $ X $.
\end{proof}


\begin{Proposition}
The map $ Un:\mathfrak{P}_{u}(\p_n(X))\rightarrow \p_n(\mathfrak{P}_{u}(X)) $ taking any family of $ n $-plots to its union is a surjective smooth map.
\end{Proposition}

\begin{proof}
By definition \ref{d:upd}, $ Un $ is surjective.
Suppose that $ \Pi:U\rightarrow \mathfrak{P}_{u}(\p_n(X)) $ is a plot in $ \mathfrak{P}_{u}(\p_n(X))  $.
To prove that $ Un\circ\Pi $ is a plot in  $ \p_n(\mathfrak{P}_{u}(X)) $,
let $ r_0\in U $ and  $ s_0\in\mathrm{dom}(Un\circ\Pi(r_0)) $. 
Notice that $ \Pi(r_0) $ is a collection of $ n $-plots in $ X $, so one can find an $ n $-plot $ P_0\in \Pi(r_0) $ such that $ s_0\in \mathrm{dom}(P_0) $.
Since $ \Pi $ is a plot in $  \mathfrak{P}_{u}(\p_n(X))  $, by definition, there exist an open neighborhood $ U' \subset U $ of $ r_0 $ and a  smooth family  $ \sigma : U' \rightarrow \p_n(X) $ of plots in $ X $ such that $ \sigma(r) \in \Pi(r) $, 
for every $ r \in U' $ and $ \sigma(r_0)=P_0 $.
For $ s_0\in\mathrm{dom}(\sigma(r_0)) $, there exist an open neighborhood $  V\subset U' $ of $ r_0 $
and an open neighborhood $  W $ of $ s_0 $ such that $ W \subset\mathrm{dom}(\sigma(r)) $ for all $ r\in V $, and $ (r,s)\mapsto \sigma(r)(s)$ defined on $ V\times W $ is a plot in $ X $.
Thus $ W \subset\mathrm{dom}(\sigma(r))\subset\mathrm{dom}(Un\circ\Pi(r))  $ for all $ r\in V $.

Now we show that $ (r,s)\mapsto Un\circ\Pi(r)(s)$ defined on $ V\times W $ is a plot in $ \mathfrak{P}_{u}(X) $. Let
$ (v_0,w_0)\in V\times W  $ and $ x_0\in Un\circ\Pi(v_0)(w_0)$. This means that $ w_0\in \mathrm{dom}(Un\circ\Pi(v_0))  $ and that
for some $ n $-plot $ P_0\in \Pi(v_0) $ we have $ P_0(w_0) =x_0 $.
Similar to the above argument, one can find a local implicit plot  $ (r,s)\mapsto \sigma(r)(s)$ in $ X $ with $ \sigma(r)(s)\in Un\circ\Pi(r)(s)$ and $ \sigma(v_0)(w_0)=P_0(w_0)=x_0 $.
Therefore, $ Un\circ\Pi $ is a plot in $  \mathcal{D}_n(\mathfrak{P}_{u}(X))  $.
\end{proof}

\begin{Cor}
The supremum map $ Sup:\mathsf{Comp}(\mathfrak{P}_{u}(\mathcal{D}_n(X)))\rightarrow \mathcal{D}_n(X) $ taking any compatible family of $ n $-plots to its supremum is smooth.
\end{Cor}

\section{The strong power set diffeology}\label{spsdiff}
Hereafter, we have another diffeology on the power set, which  we call  the strong power set diffeology.

\begin{Definition} \cite[p. 61]{Igdiff}  
Let $ X $ be a diffeological space. The set of all parametrizations
$ P:U\rightarrow\mathfrak{P}(X) $
with the following property is a diffeology on $  \mathfrak{P}(X) $, called the \textbf{strong power set diffeology}: 
\begin{enumerate}
\item[\textbf{SPD.}] 
for every $ r_0 \in U $ and every plot $ Q_0  $ in $ X $ with $ \mathrm{Im}(Q_0)\subset P(r_0) $, there exist an open neighborhood $ V \subset U $ of $ r_0 $ and a \textit{local smooth family} $ \sigma : V \rightarrow \mathcal{D} $ of plots  in $ X $  such that $ \mathrm{Im}(\sigma(r)) \subset P(r) $, 
for every $ r \in V $ and $ \sigma(r_0)=Q_0 $.
\end{enumerate} 	
We denote by $  \mathfrak{P}_{s}(X) $  the power set	$  \mathfrak{P}(X) $ endowed with the strong power set diffeology and call it the \textbf{strong power space}.
\end{Definition}
The strong power set diffeology is indeed the coarsest diffeology on  $  \mathfrak{P}(X) $ such that the map
$ \delta_X:\mathfrak{P}_{s}(X)\rightarrow \mathfrak{P}_{u}(\mathcal{D}) $ taking any subset $ A\subset X $ to the subspace diffeology on $ A $, i.e.
$ \delta_X(A)=\{P\in\mathcal{D}\mid \mathrm{Im}(P)\subset A \} $,
is smooth (see \cite[Exercise 62(2)]{Igdiff}).
One can see that $ \delta_X $ is an induction.

\begin{Proposition}\label{pro:strong-union}
The identity map 
$ \mathrm{id}:\mathfrak{P}_{s}(X)\rightarrow \mathfrak{P}_{u}(X) $
is smooth. That is, the strong power set diffeology is finer than the union power set diffeology.
\end{Proposition} 
\begin{proof}
Let 	$ P:U\rightarrow\mathfrak{P}_{s}(X) $ be a plot and let  $ r_0 \in U $, $ x\in P(r_0) $.
Consider the $ 0 $-plot $ \mathbf{x} $ for which we get $ \mathrm{Im}(\mathbf{x})\subset P(r_0) $.
By \textbf{SPD}, we obtain  an open neighborhood $ V \subset U $ of $ r_0 $ and a  local smooth family $ \sigma : V \rightarrow \mathcal{D} $ of plots  in $ X $  such that $ \mathrm{Im}(\sigma(r)) \subset P(r) $, 
for every $ r \in V $ and $ \sigma(r_0)=\mathbf{x} $.
By definition, there exist an open neighborhood $  V'\subset V $ of $ r_0 $
such that $  \sigma(r) $ is a $ 0 $-plot for all $ r\in V' $, and $ r\mapsto \sigma(r)(0) $ defined on $ V' $ is a plot in $ X $.
Now  define $ \tau: V'\rightarrow X $ by $ \tau(r)=\sigma(r)(0) $ so that $ \tau(r)\in P(r) $ and $ \tau(r_0)=\sigma(r_0)(0)=x $.
Therefore, $ P $ satisfies \textbf{UPD} and it is a plot in $ \mathfrak{P}_{u}(X) $.
\end{proof}
\begin{rem}\label{rem:strong-lsc}
By Theorem \ref{The-lsc}, and the fact that the identity map 
$ \mathrm{id}:\mathfrak{P}_{s}(X)\rightarrow \mathfrak{P}_{u}(X) $
is smooth, 
any  smooth set-valued map $ \phi:Y\rightarrow\mathfrak{P}_{s}(X) $ is lower semi-continuous with respect to the D-topology.
\end{rem}
In \cite{DP}, it is shown that the map $ \imath:X\rightarrow\mathfrak{P}_{s}(X) $ is an embedding. Furthermore, we can say it is actually a strong embedding.
\begin{Proposition}
The map $ \imath:X\rightarrow\mathfrak{P}_{s}(X) $ defined by $ \imath(x)=\{x\} $ is a strong embedding.
\end{Proposition}
\begin{proof}
In view of Proposition \ref{pro:strong-union} and Remark \ref{rem:strong-lsc}, the proof is quite similar to that of Proposition \ref{pro:union-se}.
\end{proof}

\section{Projectable diffeologies on power sets}\label{ppsdiff}

\subsection{The locally projectable parametrizations on $\mathfrak{P}^\star(X)$}
By contrast with global projectable parametrizations, we define local projactable parametrizations at a fixed $A \in \mathfrak{P}^\star(X).$
\begin{Definition}
Let $A \in \mathfrak{P}^\star(X)$ and let $U$ be a non-empty open subset of an Euclidean space. A {\bf local projectable parametrization} at $A$ with domain $U$ is a map $ \phi: U \rightarrow \mathfrak{P}^\star(X)$ such that there exists a smooth map $\varphi : U \rightarrow C^\infty(A,X)$ such that 
\begin{itemize}
\item  $\varphi(r_0)|_{A} = \mathrm{id}_A$, for some $ r_0 \in U,$
\item $\phi(r) = \varphi(r)(A),$  for all $r \in U$.
\end{itemize}
and we define the \textbf{locally projectable diffeology} $\p_{lp}$ as the diffeology generated by the family of local projectable parametrizations for each $A \in \mathfrak{P}^\star(X).$ 
\end{Definition}
Since $C^\infty(A,X)$ is equipped with the functional diffeology, the family of local projectable parametrizations has the covering property (1) and the smooth compatibility property (3) of Definition \ref{d:diffeology}. In fact, it constitutes an example of a prediffeology (see \cite[Definition 2.6]{DA}). The locally projectable parametrizations do not form a diffeology in most case. This explains the  formulation for $\p_{lp}.$ The power set $\mathfrak{P}(X),$ equipped with $\p_{lp},$ is denoted by $\mathfrak{P}_{lp}(X).$ 

\begin{Proposition} 
The identity map 
$ \mathrm{id}:\mathfrak{P}^\star_{lp}(X)\rightarrow \mathfrak{P}^\star_{u}(X) $
is smooth, namely, the locally projectable diffeology is finer than the union power set diffeology.
\end{Proposition} 
\begin{proof}
It is sufficient to prove that any local projectable parametrization is a plot in $ \mathfrak{P}^\star_{u}(X) $, because then the diffeology generated by the family of local projectable parametrizations would be contained in $ \mathfrak{P}^\star_{u}(X) $.
Let $ A \in \mathfrak{P}^\star(X)$ and suppose that $ \phi: U \rightarrow \mathfrak{P}^\star(X)$ is a local projectable parametrization at $ A $, $ s_0\in U $, and $ x_0\in\phi(s_0) $. By definition, there exists a smooth map $\varphi : U \rightarrow C^\infty(A,X)$ such that $\phi(r) = \varphi(r)(A) $  for all $r \in U$.
In particular,  $x_0\in\phi(s_0) = \varphi(s_0)(A) $, so there exists an element $ a\in A $ with $ \varphi(s_0)(a)=x_0 $.
Define $ \sigma:U\rightarrow X $ by $ \sigma(r)=\varphi(r)(a) $, which is smooth. Moreover, $ \sigma(r)=\varphi(r)(a)\in \varphi(r)(A)=\phi(r) $. And, $ \sigma(s_0)=\varphi(s_0)(a)=x_0 $. Thus, $ \phi: U \rightarrow \mathfrak{P}^\star(X)$ is a plot in $ \mathfrak{P}^\star_{u}(X) $.
\end{proof}
Analogous to the case of the strong power set diffeology, one can observe that
every smooth set-valued map $ \phi:Y\rightarrow\mathfrak{P}_{lp}(X) $ is lower semi-continuous with respect to the D-topology.
Also, the map $ \imath:X\rightarrow\mathfrak{P}_{lp}(X) $ defined by $ \imath(x)=\{x\} $ is a strong embedding.

\subsection{The globally projectable parametrizations on $\mathfrak{P}(X)$} \label{gp}
The monoid of smooth maps $C^\infty(X,X)$ equipped with the composition rule has a natural action on the left on $\mathfrak{P}(X)$ by 
\begin{equation}
\begin{array}{cccc}
L : & C^\infty(X,X) \times \mathfrak{P}(X) &\rightarrow &\mathfrak{P}(X) \\
& (f, A) & \mapsto & f(A)
\end{array} 
\end{equation}
\begin{Proposition}
$L$ is smooth when
\begin{itemize}
\item $C^\infty(X,X)$ is equipped with the functional diffeology,
\item $\mathfrak{P}(X)$ is equipped with either the weak power set diffeology, or the strong power set diffeology.
\end{itemize}
\end{Proposition}
\begin{proof}
Each case follows from the definition of the functional diffeology.
\end{proof}
We note by $L_D$ the restriction of $L$ to $\mathrm{Diff}(X)$.
\begin{Proposition}
$L_D$ is smooth when
$\mathfrak{P}(X)$ is equipped with the weak power set, the union power set  or the strong power set diffeology.
\end{Proposition}
\begin{proof}
Each case follows from the definition of the functional diffeology.
\end{proof}
\begin{rem}[$C^\infty(X,X)$ or $\mathrm{Diff}(X)$?] When $X$ is a smooth compact manifold, $\mathrm{Diff}(X)$ is an open subset of $C^\infty(X,X)$ 
for its classical structure of Fr\'echet manifold, and hence in particular for the D-topology of the functional diffeology.
Therefore one can wonder whether, in the definitions, the choice of the use of the \textbf{group} $\mathrm{Diff}(X)$ would be more accurate  than the use of the \textbf{monoid} $C^\infty(X,X).$ Analyzing, e.g., \cite{Olv,Rob} where projectable symmetries are highlighted, the choice of $\mathrm{Diff}(X)$ may be justified by the definition of symmetries for functions on diffeologies, along the lines of \cite{Ma2020-3}. However, (diffeologically) projectable transformations must be generated by (smooth) transformations on $X,$ and the best space for this purpose is $C^\infty(X,X)$ in which $\mathrm{Diff}(X)$ is not an open subset. Therefore, out of precise technical statements comparing these two sets, we feel the necessity to consider the biggest one. However, this third projectable-type maps define a diffeology on power sets which has stronger properties than others.
\end{rem}

\begin{Definition}
We define \textbf{globally projectable parametrizations} of $ \mathfrak{P}^\star(X)$ the push-forward parametrization $\phi$ such that there exists locally a smooth map $P: U \rightarrow C^\infty(X,X) $ and $A \in \mathfrak{P}^\star(X)$ such that the restriction of $\phi$ to the domain $U$ can be defined by $P(\cdot)(A).$ These parametrizations define a diffeology that we denote by $\p_{gp}.$ The same definitions hold changing $C^\infty(X,X)$ into $\mathrm{Diff}(X)$, and we denote this second diffeology by $\p_{gp,D}.$ 
\end{Definition}

\begin{rem}
When $X$ is a compact finite dimensional manifold, a consequence of Nash-Moser inverse functions theorem is that $\mathrm{Diff}(X)$ is open in $C^\infty(X,X).$ However, even when $X=S^1 = \R/\Z, $ $\p_{gp} \neq \p_{gp,D}.$ Indeed, consider the 1-plot $$t \in (-0.25,0.25) \mapsto [-2|t|,2|t|].$$ Then its value at $t= 0$ shows that this is not the orbit of a diffeomorphism while this is possible to find a smooth map $\varphi: (-0.25,0.25) \times S^1 \rightarrow S^1, $ and hence a map $\Phi: (-0.25,0.25) \rightarrow C^\infty(S^1,S^1)$ defined by $\phi(t) = (x \mapsto \varphi(t,x))$ which globally parametrize this plot.   
\end{rem}

\begin{Theorem}
In $\p_{gp,D},$ $\overline{(.)}$ is smooth, and $\emptyset$ as well as $X$ are isolated.
\end{Theorem}
\begin{proof}
The proof follows from the fact that for any $g \in \mathrm{Diff}(X),$ and for any $A \in \mathfrak{P}(X),$ $\overline{g(A)}= g (\overline{A}).$
\end{proof}

\begin{rem}
$\cup$ and $\cap$ are \textit{not} smooth in $\p_{gp,D}.$ Let us give a counter-example. Let $X=\R,$ let $A=B=\{0\}$ and let $g_1(t): x \mapsto x+t,$ $g_2(t): x \mapsto x-t.$ Let $P_1(t) = g_1(t)(A)$ and let $P_2(t) = g_2(t)(A).$ We have that $(P_1,P_2)\in \p_{gp,D}$ but
$$ P_1(t) \cap P_2(t) = \left\{ \begin{array}{l} \emptyset \hbox{ if } t \neq 0 \\
\{0\} \hbox{ if } t = 0 \end{array}\right.$$
Since $\emptyset$ is isolated in $\p_{gp,D},$ $t \mapsto P_1(t) \cap P_2(t) \notin \p_{gp,D}.$
The same arguments hold for $t \mapsto \overline{ P_1(t)} \cup \overline{P_2(t)} \notin \p_{gp,D}.$
\end{rem}

\section{Diffeologies on Borel algebras and measures}

\subsection{Boolean diffeologies on a Boolean algebra} \label{bdiff}

Let ${ \mathcal{A} }\subset \mathfrak{P}(X)$ be a Boolean algebra. We denote by $\mathcal{A}^\star = \mathcal{A }- \{\emptyset\}.$	
\begin{Definition} \label{BDstar}
Let $\p$ be a diffeology on $\mathcal{A}^\star.$ Then
\begin{enumerate}
\item[(1)] $\p$ is $\cup-$\textbf{stable} if the map $ \cup: (\mathcal{A}^\star)^2 \rightarrow \mathcal{A}^\star$ is smooth
\item[(2)] $\p$ is \textbf{complement-stable} if the map $ \overline{(.)}: (\mathcal{A}^\star-\{X\}) \rightarrow \mathcal{A}^\star(X)-\{X\}$ is smooth.
\item[(3)] $\p$ is \textbf{boolean} if (1) and (2) are fulfilled.
\end{enumerate}

\end{Definition}
With such a definition, a superficial reader may think that, as a direct consequence, all boolean operations such as intersection, symmetric difference are smooth if the diffeology is boolean. We here have to care about $\emptyset$ which is not in the basic target space $\mathfrak{P}^\star(X)$ of our plots. Again, we make this distinction because of our examples produced in next section, on which we may just say naively that $\emptyset$ is disconnected (diffeologically) from $\mathfrak{P}^\star(X).$ Let us consider the following example, which will fit with all our next diffeologies, and which will explain in few words that this choice is not suitable with the classical intuition. 
\begin{example}
Let $t \in \R$ and let $c(t)=[t-1,t+1] \in \mathfrak{P}^\star(\R)${ be a ``traveling interval'' path}. Let $\gamma(t) = c(t) \cap c(-t).$ Then $\gamma(t) = \emptyset$ if $t \notin [-1,1].$ 
\end{example}
Therefore, we have to produce diffeologies such that $\{\emptyset\}$ is not necessarily a disconnected component of $\mathfrak{P}(X).$

\begin{Theorem}
Let $\p$ be a boolean diffeology on $\mathcal{A}^\star.$ Then $\p \cup \overline{\p}$ generates local plots of a diffeology $\p^b$ on $\mathcal{A}$ for which boolean operations are smooth. We also call this diffeology a boolean diffeology (on $\mathcal{A}$).
\end{Theorem} 

\begin{proof}
Let us first describe the plots at $X.$ These are the maps $P: U \rightarrow \mathcal{A}^\star, $ where $ U $ is an open subset of an Euclidean space, for which there exists an open cover $\{U_i, i \in I\}$ of $U $ and a family of plots $\{\phi_i,\varphi_i, i \in I \}$ such that: 
\begin{itemize} 
\item $\forall i \in I, \mathrm{dom}(\phi_i)=\mathrm{dom}(\varphi_i) = U_i$
\item $\forall i \in I, P|_{U_i} = \phi_i \cup \varphi_i.$ 
\end{itemize}
Since $\p$ is boolean, then $\overline{(.)}\circ  \phi_i$ and  $\overline{(.)}\circ  \phi_i$ are in $\overline{\p},$ which shows that $\overline{(.)}\circ  P$ defines a parametrization at $\{\emptyset\}$ such that, if $W =  \left(\overline{(.)}\circ  P\right)^{-1}(\mathfrak{P}^\star(X)),$ $W$ is open in $U$ and  $\left(\overline{(.)}\circ  P\right)|_{W} \in \p.$ Therefore, $\p$ is the subset diffeology on $\mathcal{A}^\star$ of the diffeology generated by $\p \cup \overline{\p}$ on $\mathcal{A}^\star.$ Moreover, it is clear that this diffeology $\p^b$ is now stable under $\cup$ and $\overline{(.)}$, and that these operations are smooth. Therefore, the other two classical operations, the intersection $\cap$ and the symmetric difference $\Delta$ are also smooth. 
\end{proof}
Therefore, for any Boolean diffeology on $\mathcal{A}^\star$ in the sense of Definition \ref{BDstar} we are able to build up a diffeology on $\mathcal{A}$ such that: 
\begin{itemize}
\item the Boolean operations $\cup, \cap, \overline{(.)}$ and $\Delta$ are smooth, 
\item this diffeology coincides with the initial diffeology on $\mathfrak{P}^\star(X) - \{X\}$
\item $\{\emptyset\}$ is not a priori a disconnected component.
\end{itemize} 
We call it \textbf{Boolean diffeology} on $\mathcal{A},$ vocabulary that also applies to the maximal Boolean subalgebra $\mathfrak{P}(X).$ This construction will be discussed 
more extensively in section \ref{diffmeas} but in a more general way, one has to know whether such a Boolean diffeology exists, because our construction starts from a Boolean diffeology on $\mathfrak{P}^\star(X),$ which is assume to be at hand. 

For this, let us consider not only $\mathfrak{P}(X)$ but also any boolean algebra $\mathcal{A} \subset \mathfrak{P}(X).$
\begin{Lemma} \label{full}
Let $\mathcal{A}$ be a Boolean algebra on $X.$ 
The  indiscrete diffeology, that is, the diffeology defined by all maps
$$ U \rightarrow \mathcal{A},$$ where $U$ is an open subset of an Euclidean space, is a Boolean algebra. 
\end{Lemma}
The proof is direct. Therefore, one can wish to ``complete'' a diffeology to a Boolean diffeology.
\begin{Theorem}
Let $\mathcal{A}$ be a Boolean algebra equipped with a diffeology $\p.$
Then, there exists a unique Boolean diffeology, minimal for inclusion, which contains $\p.$ 
\end{Theorem}
\begin{proof}
Let $\mathcal{S}(\p)$ be the set of {Boolean} diffeologies that contain $\p.$ $\mathcal{S}(\p)\neq \emptyset$ by Lemma \ref{full}. Moreover, if a map $f: X \times X \rightarrow X$ is smooth for two different diffeologies $\p_1$ and $\p_2$ on $X,$ then it is smooth for $\p_1 \cap \p_2,$ then $\mathcal{S}(\p)$ is stable for intersection. Therefore, for the partial order $\subset$ on $\mathcal{S}(\p),$ applying Zorn's Lemma, we get that $\mathcal{S}(\p)$ has a unique minimal element for $\subset.$ \end{proof}
Hence, we can state:
\begin{Lemma}
There exists a minimal Boolean diffeology on $\mathfrak{P}(X)$ that contains $\p_{gp,D}.$
\end{Lemma} 

A more ``concrete'' example of a Boolean diffeology on a Boolean algebra, which is not the minimal one, will be given in section \ref{diffmeas}.

\subsection{Diffeologies on a Borel algebra}\label{diffmeas}

Let us consider a topology $\tau$ on the diffeological space $(X,\p_X)$ such that any plot $P \in \p_X$ is continuous with respect to $\tau,$ and let $\mathcal{B}(\tau)$ its associated $\sigma-$algebra. 
Let $\mathcal{B}^\star(\tau) = \mathcal{B}(\tau) - \{\emptyset\} $.
Then, for each diffeology $\mathcal{D}_{\mathfrak{P}}$ on $\mathfrak{P}(X)$ (resp. $\mathfrak{P}^\star(X)$),  $\mathcal{B}(\tau)$ (resp. $\mathcal{B}^\star(\tau)$) can be equipped 
with the subset diffeology, as well as $\mathcal{B}^\star(\tau)-\{X\}$. Since this carries no ambiguity in this section, we note them all by $\p_\mathcal{B}.$
\begin{Definition} \label{d:smooth-finite-mu}
\begin{itemize}
\item Let $\mu$ be a finite Borel measure on $X.$ Then $\mu$ is called \textbf{smooth} if it is a smooth map  for the diffeology $\p_\mathcal{B}.$  
\item Let $\mu$ be a Borel measure on $X.$ Then $\mu$ is \textbf{smooth} if $$\forall A \in \mathcal{B}(\tau), \quad 0<\mu(A) < +\infty \quad\Longrightarrow\quad \mu_A=\mu(\cdot \cap A) \hbox{ is smooth}.$$
\end{itemize}
\end{Definition}

\begin{Theorem}
Let $\tau$ be a topology on a diffeological space $(X,\p_X)$ such that any plot $P \in \p_X$ is continuous with respect to $\tau.$
Let $\p_\mathcal{B}$ be $\cup-$stable diffeology on $\mathcal{B}(\tau).$ Let $Mes(\p_\mathcal{B})$ be the space of smooth measures with respect to $\p_\mathcal{B^\star}.$ 
Then, if $Mes(\p_\mathcal{B}) \neq \emptyset,$
\begin{enumerate}
\item $\forall \mu \in Mes(\p_\mathcal{B}),$ the maps
$$ (A,B) \in \mathcal{B}^\star(\tau)^2 \mapsto \mu(A \cup B),$$
$$ A \in \mathcal{B}^\star(\tau) - \{X\}\mapsto \mu(\overline{A}) $$
are smooth.
\item Therefore, $\p_\mathcal{B}$ generates a Boolean diffeology on $\mathcal{B}(\tau)$ following the procedure from section \ref{bdiff}, and 
\item The set of finite smooth measures for $\p_\mathcal{B}$ is a generating family for a boolean Fr\"olicher structure, i.e. a boolean reflexive diffeology, on $\mathcal{B}(\tau).$
\end{enumerate}
\end{Theorem}


\begin{proof}
If $\p_\mathcal{B}$ is $\cup-$stable, $\forall (P_1,P_2) \in  \p_\mathcal{B},$
if $P_1: U_1 \rightarrow \p_B, $ if $P_2 :U_2 \rightarrow \p_\mathcal{B}, $ 

\begin{itemize}
\item $P_1 \cup P_2 : U_1 \times U_2 \rightarrow \mathcal{B} \in \p_\mathcal{B}$ and hence,  following Definition \ref{d:smooth-finite-mu}, $\forall A \in \mathcal{B}(\tau),$ such that  $0<\mu(A) < +\infty,$ then $\mu_A \circ (P_1 \cup P_2)$ is smooth. 

\item The same way, $\mu_A \circ \overline{(.)} \circ P_1 =   \mu(A)  - \mu_A \circ P_1$ is smooth. and this relation extends to the induced Boolean diffeology on $\mathcal{B}(\tau).$
\item Since $\mu$ is a measure, with the obvious notations, $\mu_A \circ (P_1 \cap P_2) = \mu_A \circ P_1 + \mu_A \circ P_2 - \mu_A \circ (P_1 \cup P_2)$ is smooth
and  hence we can prove smoothness of $\Delta$ by considering ${\mu_A \circ (P_1 \Delta P_2)} = \mu_A \circ (P_1 \cup P_2) - \mu_A \circ (P_1 \cap P_2)$  which shows that $\Delta$  is smooth.
\end{itemize}
Thus, $Mes(\p_\mathcal{B})$ will serve as a generating set of functions for a Fr\"olicher structure on $\mathcal{B}(\tau)$ for which each Boolean operation is smooth.
\end{proof}
Hence, we have constructed two other diffeologies from a given $\cup-$stable one on $\mathcal{B}(\tau)^\star:$ a Boolean diffeology on $\mathcal{B}(\tau)$ and its reflexive completion.
Let us now turn to a practical example.
\begin{rem} \label{vaguediff}
Let $M$ be a non-compact, smooth Riemannian manifold equipped with its Riemannian measure $\lambda$ and let $\mathcal{B}$ be its (classical) Borel algebra. 
Let $C_c^{\infty}(M)$ be the set of smooth, real-valued and compactly supported functions on $M.$ Then the evaluation maps
$$ ev_f : A \in \mathcal{B} \mapsto  \int_B f d\lambda,$$
for $f \in C_c^{\infty}(M),$ define a Fr\"olicher structure on   $\mathcal{B}.$ This Fr\"olicher structure is not $\cup-$stable. Indeed, consider $M=\R,$ $P_1(t) = (-\infty,t+0.5)$ and $P_2(t) =(-t + 0.5,+\infty),$ and a test function $f$ which is a mollifier at $0 $ with support $[-1,1],$ then it is easy to check that
\begin{itemize}
\item if $t>0,$ $ev_f(P_1(t)\cup P_2(t))= ev_f(\R)=1$
\item $$\lim_{t \rightarrow 0^-} \frac{ev_f(t) - ev_f(0)}{t}= f(0,5) \neq 0,$$
\end{itemize} which shows that the $\cup-$operation is not smooth. Therefore, it can be completed to a  $\cup-$stable diffeology, from which we can define a Boolean Fr\"olicher structure on $\mathcal{B}(\tau),$ but evaluations via smooth maps in $C_c^{\infty}(M)$ will no longer all be smooth with respect with this new diffeology. 

In a second approach, consider all the $\cup-$stable diffeologies on $\mathcal{B}^\star(\tau)$ such that the maps
$$\forall n \in \N^*, (A_1,...A_n)\in \mathcal{B}^\star(\tau)^n \mapsto  \int_{\cup_{i \in \N_n}A_i} f d\lambda$$   are smooth. 
We know that the discrete diffeology on $\mathcal{B}^\star(\tau)$ fulfills this condition, but we actually have no result to prove the existence of a maximal $\cup-$stable diffeology that makes these maps smooth.
\end{rem}

Therefore, in order to get a ``reasonable'' diffeology, one has to adapt the setting, see section \ref{measdiff} 
\subsection{Hadamard and Fomin differentiability on a vector space: the diffeological viewpoint} \label{Fom}
Let us now assume that $X$ is a diffeological vector space, equipped with a topology $\tau$ and associated Borel algebra $\mathcal{B}(\tau)$ as before. 
We now examine differentiation of measures with respect to a non zero vector $v \in X.$
Let $<v>$ be the vector space spanned by $v,$ which we identify with its one-parameter $\mathrm{Diff}(X)-$subgroup of transformation. 
We mimic the construction of the globally projectable diffeology $\p_{gp,D}$ by replacing the group of diffeomorphisms $\mathrm{Diff}(X)$ by $<v>.$

\begin{Theorem}
The family
$$\p_v = \left\{ A + (f\circ P)v \, | \, (A,P,f) \in \mathfrak{P}(X)\times\p_\infty(\R)\times C^\infty(\R,\R) \right\}$$
form a diffeology on $\mathfrak{P}(X)$ 
\begin{itemize}
\item which is complement-stable
\item for which $\emptyset$ and $X$ are isolated.
\end{itemize}
\end{Theorem}

\begin{Definition}
Let $\mathcal{A} \subset \mathfrak{P}^\star(X)$ be a space of sets stable under the $<v>-$action on $\mathfrak{P}^\star(X).$ Let $V$ be a complete locally convex topological vector space. The functional $\Phi: \mathcal{A} \rightarrow V$ 
\begin{itemize}
\item is $v-$smooth if $\Phi$ is diffeologically smooth, when $V$ is equipped with its nebulae diffeology and when $\mathcal{A}$ is equipped with its subset diffeology induced by $\p_v.$
\item is $v-$differentiable at $A \in \mathcal{A}$ if $$D_{A,v}\Phi = \lim_{t \rightarrow 0^+} \frac{\Phi(A+tv) - \Phi(A)}{t}$$ exists in $V.$
\end{itemize}

\end{Definition}

The second item of this definition is a specialization of Definition \ref{c-diff} for the path $c(t) = A + tv,$ which is smooth in $\p_v$ as well as in all the diffeologies defined in section \ref{s:Pstar}.
\begin{example}
A finite Borel measure $\mu$ on $X$ is Fomin-differentiable in the direction $v$ if and only if for each Borel set $ A $, $$\lim_{t \rightarrow 0}\frac{\mu(A+tv) - \mu(A)}{t}$$ exists (see, e.g., \cite{Bog}). Then this relation fits the fact that $$\forall A, \quad D_{A,v}\Phi = - D_{a,-v}\Phi,$$ and if the measure $\mu$ is infinitely differentiable in the sense of Fomin, then it is easy to see that the measure $\mu$ is $v-$smooth in our terminology, setting $\mathcal{A}$ as the Borel algebra.  
\end{example}
We can extend this last example in the following way. Let $\p$ be a diffeology on a Borel algebra $\mathcal{A}$ on a set $X.$
\begin{Definition}
Let $ k \in \N\cup \{\infty\}.$
We denote by $\mathbb{M}_\p^k(X)$ the space of measures $\mu$ on $\mathcal{A}$ that are $k-$differentiable  maps from $\mathcal{A} - \{\emptyset,X\}$ to $\R_+.$    
\end{Definition}

\subsection{Diffeology on spaces of measures} \label{measdiff}
Let us now consider $\mathbb{M}_+(X)$ the space of ($\sigma-$finite) measures on a diffeological space $X$ equipped with a topology $\tau$ and its associated Borel algebra $\mathcal{B}(\tau)$ as before. In the test space
$$\left\{ev_A: \mu \in \mathbb{M}_+(X) \mapsto \mu(A) \, | \, A \in \mathcal{A} \right\}$$
the mappings $ev_A$ are with values in $\R_+ \cup \{+\infty\}.$
Moreover, on $\R_+,$ all smooth paths $\gamma$ passing through $0$ have a vanishing Taylor development on $\gamma^{-1}(0).$ Therefore, in order to avoid the same problems as in Remark \ref{vaguediff} we change the considered diffeology  on each interval $\R_+$  for another diffeology which is the following:

\begin{Definition}
The \textbf{piecewise smooth diffeology} on $\R$ is defined by the set of piecewise smooth parametrizations of $\R.$ We note this diffeology by $\p_p(\R).$ With the same notations, we also define $\p_p(I)$ where $I$ is a closed bounded interval.

Moreover, we define the \textbf{piecewise smooth diffeology} on $\R \cup \{+ \infty\}$ by $$\p_p(\R_+ \cup \{+ \infty\})= argth^* \left( \p_p([0,1])\right).$$
\end{Definition}

\begin{Definition}
We define $$\p(\mathbb{M}_+(X)) = \bigcap_{A \in \mathcal{A}} ev_A^*\left(\p_p(\R_+ \cup \{+ \infty\})\right).$$
\end{Definition}

This diffeology is difficult to study, but we can remark the following: 
\begin{rem}
Following Remark \ref{vaguediff} and considering the two paths $P_1$ and $P_2,$ $P_1 \cup P_2 \in \p(\mathbb{M}_+(X)) $ and $\p_1 \cap \p_2 \in \p(\mathbb{M}_+(X))$ are smooth.

\end{rem}

\section{Set-valued differentials}\label{stvdiff}

The concept of smoothness in diffeology does not imply most classical properties of \textbf{derivatives}, which highly depend on the underlying diffeology. Indeed, only the \textbf{differential} of a smooth function, in the sense of differential forms, is actually well-established and commonly used. Let us give highlights of the problem of the derivative in diffeology through a generalization of set-valued derivatives.

We fix $(X,\p)$ a diffeological space an a functional $$\Phi: X \rightarrow V$$ where $V$ is a complete locally convex vector space. The functional $\Phi$ is not a priori assumed to be smooth. 

\begin{Definition} \label{c-diff}
Let $x \in X$ and let $c \in C^\infty(\R,X)$ such that $c(0)= x.$ Then the functional $\Phi$ is $c-$differentiable at $x$ if the limit 
$$ D_{x,c} \Phi = \lim_{t \rightarrow  0^+} \frac{\Phi(c(t)) - \Phi(x)}{t}$$
exists in $V.$
\end{Definition}

Classical examples show that the definition of the differentiability with respect to a path $c$ does not only depend on the germ of the path $c$ at $x.$ 

\begin{example}  \label{multi-deriv} Let $X = \R^2$ equipped with its usual diffeology. Let $$\Phi(x,y) = \left\{ \begin{array}{lcr} 
\frac{xy}{|y|} & \quad & \hbox{if } y \neq 0 \\
0 & \quad & \hbox{if } y = 0
\end{array}\right.$$ 
Let $c_\alpha: t \in \R \mapsto c_\alpha(t) = (t, \alpha t^2)$ for $\alpha \in \R.$ All paths $c_\alpha$ have the same germ at $t = 0$ but $ D_{0,c_\alpha} \Phi \in \{-1,0,1\}$ depending on the value of $\alpha.$     
\end{example}
Hence, we define a set-valued version of Hadamard differentiation. We note by $germ(c)$ the germ of a path $c.$

\begin{Definition}
Let $u \in {}^iT_xX.$ 
Then $\Phi$ is differentiable at $x$ in the direction $v$ if $\forall c \in \mathcal{C}_x$ such that $v = germ(c),$ $\Phi$ is $c-$differentiable at $x$ and we define 
$$D_u \Phi (x)= \bigcup_{germ(c)=u}Adh\Big{(} n\big{(}\Phi(c(1/n))-\Phi(c(0))\big{)}\Big{)},$$
where $Adh$ means the adherence set of the sequence, when $n \rightarrow +\infty.$ 
\end{Definition} 
In this definition, the derivative is a set-valued map and $\Phi$ is differentiable in the classical sense if $D_u \Phi$ has an unique value in $V.$ This matches with our classical definition of the derivative of a differentiable function on a smooth manifold. However, for example, on an open subset $U$ of $\R^2$ we have to mention that 
\begin{itemize}
\item for the nebulae diffeology $\p_\infty(U),$ the space of germs at any point is two-dimensional, and coincides with the classical tangent space. 
\item for the spaghetti diffeology $\p_1(U),$ the tangent space at any point is of uncountable dimension \cite{CW2022}. 
\end{itemize}
In Example \ref{multi-deriv}, the family $\left\{c_\alpha; \alpha \in \R\right\}$ has only one germ at $t=0$, which confirms that $$D_{(1,0)}\Phi(0) \subset \{-1,0,1\}$$ for $\p_\infty(U)$ but for $\p_1(U),$ the same notation $D_{(1,0)}\Phi(0)$ only considers the path $c_0(t)$ and its reparametrizations and therefore $$D_{(1,0)}\Phi(0) = \{0\}.$$

\section{Diffeologies and shapes}\label{diffshp}

For a detailed introduction into shape calculus we refer to \cite{DZ2001,Sz1992}. Let $E$ be a set.
\begin{Definition}
Let $d \in \N^*$ and $D \in \mathfrak{P}^\star(\R^d).$ 
\begin{itemize}
\item A \textbf{shape space} is a set $\mathcal{A} \subset \mathfrak{P}^\star(D),$
\item A \textbf{shape functional} is a function $J: \mathcal{A} \rightarrow \R$
and an \textbf{underconstrained shape optimization problem} is given by $$\min_{\Omega \in \mathcal{A}} J(\Omega)$$ 
\item A \textbf{constrained shape optimization problem} involves a shape functional $J: \mathcal{A}\times Y \rightarrow \R,$ where $Y$ is a space of state variables, usually a function space of solutions satisfying a system of PDEs, depending on $\Omega \in \mathcal{A},$ and is given by $$\min_{(\Omega,y) \in \mathcal{A}\times Y} J(\Omega,y).$$
\end{itemize}
\end{Definition}
In order to minimize a functional $\Phi,$ initial steps of some classical methods can rely on:
\begin{itemize}
\item solving $D \Phi (x)=0$,
\item or defining a gradient $\nabla \Phi$. 
\end{itemize} 
In both cases, differentiability is the most powerful tool to deal with minimization procedures. In order to differentiate a shape functional $J$ on a shape space $\mathcal{A},$ we need to define admissible paths defined as follows:
let $F_t$ be a map $\R \rightarrow C^0(D,\R^d),$ with $F_0 =\mathrm{id}_D$ and with a priori no assumption on regularity of the $ 1 $-parameter family $\{F_t\}_{t \in \R} .$ Then, such a map defines a \textbf{flow} $$(\Omega,t) \mapsto \Omega_t = F_t(\Omega)$$ for $\Omega \in \mathcal{A}.$ In classical practice of shape analysis, the family $F_t$ is defined through a suitable vector field $x \in \R^d\mapsto V(x)$  by $$F_t(x) = x + t V(x).$$
\begin{Definition}
The \textbf{Eulerian derivative} of the shape functional $J$ at $\Omega$ in the direction $V$ is defined by $$D_VJ(\Omega) = \lim_{t \rightarrow 0^+} \frac{J(\Omega_t) - J(\Omega)}{t} .$$
\end{Definition}
Therefore, according the Eulerian derivative reads as $$D_VJ(\Omega)= D_{\Omega,v} J$$ with $v(t) = \bigcup_{x \in \Omega} \{x+tv(x)\}.$ The map $t \mapsto (x \mapsto x + tV(x))$ defines a smooth path on $C^\infty(X,X)$ and hence $v \in \p_{gp}.$

\end{document}